\documentclass[11pt,regno]{amsart}

\usepackage{graphicx}
\usepackage[rgb]{xcolor}
\usepackage{amsmath,amscd,amsthm,amssymb,tikz} 
\usetikzlibrary{calc,positioning}

\usepackage{overpic}

\usepackage{euscript,graphicx,color}
\usepackage{amsmath}
\usepackage{amssymb}
\usepackage{amsfonts}
\usepackage{mathrsfs}

\newtheorem{theo}{Theorem}[section]
\newtheorem{theo-app}{Theorem}[section]

\newtheorem{theorem}[theo]{Theorem} 

\newtheorem{proposition}[theo]{Proposition}

\newtheorem{lemma}[theo]{Lemma}
\theoremstyle{definition}

\newtheorem{remark}[theo]{Remark}
\newtheorem*{remark*}{Remark}
\newtheorem{definition}[theo]{Definition}
\newtheorem*{definition*}{Definition}

\theoremstyle{Theorem A}
\theoremstyle{Theorem B}
\theoremstyle{Theorem C}
\theoremstyle{Theorem D}
\theoremstyle{Theorem E}

\newtheorem*{thmA}{Theorem A}
\newtheorem*{thmB}{Theorem B}
\newtheorem*{thmC}{Theorem C}
\newtheorem*{thmD}{Theorem D}
\newtheorem*{thmE}{Theorem E}
\theoremstyle{Conjecture}
\newtheorem*{con}{Conjecture}

\newtheorem{notation}[theo]{Notation}

\theoremstyle{citing}

\numberwithin{equation}{section}

\renewcommand{\hat}{\widehat}

\DeclareMathOperator{\cl}{cl}

\DeclareMathOperator{\BD}{BD}

\newcommand{\R}{\mathbb{R}}
\newcommand{\N}{\mathbb{N}}

\newcommand{\Z}{\mathbb{Z}}

\newcommand{\D}{\mathbb{D}}

\DeclareMathOperator{\diam}{diam}

\DeclareMathOperator{\dist}{Dist}

\DeclareMathOperator{\Leb}{Leb}

\DeclareMathOperator{\Comp}{{\rm Comp}}

\DeclareMathOperator{\Const}{{\rm Const}}

\DeclareMathOperator{\irreg}{{\rm irreg}}

\DeclareMathOperator{\HD}{{\rm HD}}

\DeclareMathOperator{\reg}{{\rm reg}}

\def\bH{\mathbf{H}}
\def\bW{\mathbf{W}}

\def\cV{\EuScript{V}}
\def\cW{\EuScript{W}}

\def\cZ{\EuScript{Z}}

\def\cR{\mathscr{R}}

\newcommand{\cH}{\mathcal{H}}

\def\a{\alpha}        \def\La{\Lambda}
             
\def\e{\varepsilon}      
            
\def\la{\lambda}   

\def\ov{\overline}
\def\un{\underline}

 \def\dist{{\rm {dist}}}
            
   \def\La{\Lambda}

\author[R. Mohammadpour ]{Reza Mohammadpour} \address{Institute of Mathematics, Polish Academy of Sciences, ul. \'{S}niadeckich 8, 00-656 Warszawa, Poland}
\email{rmohammadpour@impan.pl}

\author[F. Przytycki]{Feliks Przytycki} \address{Institute of Mathematics, Polish Academy of Sciences, ul. \'{S}niadeckich 8, 00-656 Warszawa, Poland}
\email{feliksp@impan.pl}

\author[M. Rams]{Micha{\l} Rams} \address{Institute of Mathematics, Polish Academy of Sciences, ul. \'{S}niadeckich 8, 00-656 Warszawa, Poland} \email{rams@impan.pl}

\begin{document}

\date{\today}

\title[Solenoids and fractal dimensions]{Hausdorff and packing dimensions and measures for
 nonlinear transversally non-conformal thin solenoids}



\maketitle

\begin{abstract}
 We extend results by B. Hasselblatt, J. Schmeling in \emph{Dimension product structure of hyperbolic sets} (2004), and by the third author and K. Simon in \emph{Hausdorff and packing measures for solenoids} (2003),
 for $C^{1+\varepsilon}$ hyperbolic, (partially) linear
 solenoids $\Lambda$ over the circle embedded in $\mathbb{R}^3$ non-conformally attracting in the stable discs $W^s$ direction, to nonlinear ones.

 Under an assumption of transversality and assumptions on Lyapunov exponents for an appropriate Gibbs measure imposing \emph{thinness},
 assuming also there is an invariant $C^{1+\varepsilon}$ strong stable
  foliation,
  we prove that Hausdorff dimension $\HD(\Lambda\cap W^s)$
  is the same quantity $t_0$ for all $W^s$ and else ${\rm HD}(\Lambda)=t_0+1$.

  We prove also that for the packing measure $0<\Pi_{t_0}(\Lambda \cap W^s)<\infty$ but for Hausdorff measure
  ${\rm HM}_{t_0}(\Lambda\cap W^s)=0$ for all $W^s$. Also $0<\Pi_{1+t_0}(\Lambda) <\infty$ and ${\rm HM}_{1+t_0}(\Lambda)=0$.

    A technical part says that the holonomy along unstable foliation is locally Lipschitz,
    except for a set of unstable leaves whose intersection with every $W^s$ has  measure ${\rm HM}_{t_0}$ equal to 0
       and even Hausdorff dimension less than $t_0$. The latter holds due
    to a large deviations phenomenon.
 \end{abstract}

\tableofcontents

\maketitle

\section{Introduction. Statement of main results}\label{Introduction}

We consider the solid torus
$$
M=S^1 \times \D, \; \; \D=\{(y,z)\in\R^2:y^2+z^2<1\},
$$
where $S^1=\R/2\pi \Z$.

Consider a mapping $f:M\to M$, of class $C^{1+\e}$, that is with its differential being
H\"older continuous, given by the formula
\begin{equation}\label{triangular}
f(x,y,z)=(\eta(x), \lambda( x,y  )+ u(x), \nu( x,y,z ) + v(x) ),
\end{equation}
with $\lambda(x,0)=\nu(x,0,0)=0$.
Assume that $f$ has period $2\pi$ with respect to $x$ so that it is well-defined on $M$. Assume that $\eta$ has degree $d>1$.

Denote $\eta'=\frac{d\eta}{dx}, \;  \lambda'=\frac{d\lambda}{dy}, \; 
\nu'=\frac{d\nu}{dz}$. Assume that $0<\nu' < \la'<1$ and $1<\eta'<1/\la' $
(some of these inequalities will be weakened later on to inequalities between Lyapunov exponents on $\Lambda$,
that is integrals with respect to certain Gibbs measure). We could allow here $-1<\la'<0$
(the same for $\nu'$), but we assume it is positive to simplify notation. We assume also that $f:M\to M$ is injective
(using sometimes the name \emph{embedded}).

Such a solenoid in the linear case (or at least if $\eta'\equiv d$) can be called a \emph{uniformly thin solenoid}\footnote{For the definition of a thick linear solenoid, where $\eta'\la'>1$,
see e.g. \cite{Rams1}.}. Compare stronger \emph{uniform dissipation} condition in the Outline subsection.

\medskip

\begin{center}
\begin{tikzpicture}
    \draw [thick, shade, shading=radial, inner color=white, outer color=black!25] (-339:2.56) to [out=153,in=3] (-260:1.44) to [out=185,in=60] (-190:2.70) to [out=-120,in=175] (-120:1.7) to [out=-5,in=200] (-30:2.04) to [out=25,in=-80] (13:3.02) to [out=100,in=-2] (90:2.3) to [out=178,in=92] (180:3.7) to [out=-83,in=150] (220:2.8) to [out=-20,in=190] (315:2.62) to [out=20,in=-120] (345:3.3) to [out=55,in=-40] (373:3.02) to [out=-85,in=80] (364:2.95) to [in=60,out=-50] (350:3.05) to [in=10,out=-125] (315:2.2) to [in=-20,out=188] (220:2.4) to [in=-110,out=155] (170:3.3) to [in=-180,out=62] (90:1.94) to [in=130,out=-5] (-335:2.54) to [in=95,out=-60] (-348:2.54) to [in=25,out=-80] (-30:1.56) to [in=-5,out=200] (-130:1.38) to [in=-125,out=175] (-190:2.3) to [in=181,out=50] (-270:1.09) to [in=150,out=0] (-349:2.538) to [out=90,in=-60] (-335:2.54);
    \begin{scope}[scale=2.05]
    	\draw (0,0) circle (2 and 1.35);
        \draw [scale=0.85] (-1.1,0.09) to [bend right=40] (1.1,0.09);
        \draw [scale=0.85] (-1.01,0.02) to [bend left=41] (1.01,0.02);	
        \draw (0,1.35) -- (0,1.75) -- (-0.5,1.5) -- (-0.5,0) -- (0,0.25) -- (0,0.35);
        \draw (0,0.35) -- (0,1.35);
        \draw (-0.25,0.831)+(90:0.508) arc (90:270:0.2 and 0.508);
        \draw [dashed] (-0.25,0.831)+(-90:0.508) arc (-90:90:0.2 and 0.508);
	\end{scope}
    \draw [fill=black!50, xshift=-0.5cm,yshift=2.11cm,rotate=29] (0,0) circle (0.26 and 0.16);  
    \draw [fill=black!50, xshift=-0.5cm,yshift=1.238cm,rotate=27] (0,0) circle (0.22 and 0.13);
\end{tikzpicture}
\end{center}

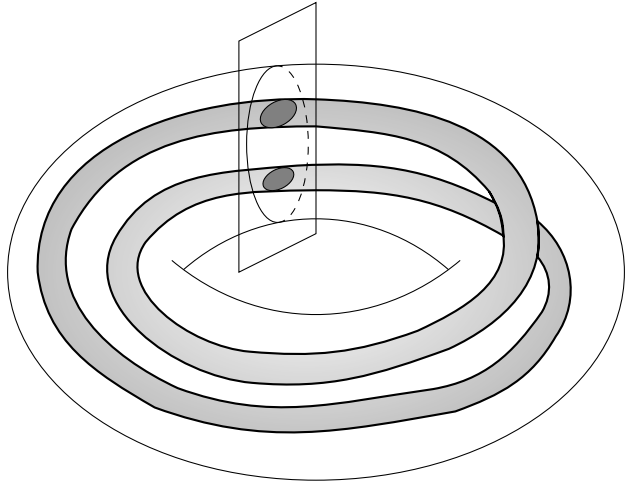
FIGURE 1. Solenoid.

\bigskip

Then
$$
\Lambda:=\bigcap_{n=0}^\infty f^n(M)
$$
is an invariant hyperbolic set, so called expanding attractor.
The assumption $f$ is injective on $M$ can be weakened to the assumption
$f$ is injective on $\Lambda$, by replacing $M$ by a solid torus being a sufficiently thin neighbourhood of $\Lambda$.
However, for clarity, we assume the injectivity of $f$ directly on $M$.

\medskip

For each $p=(x,y,z)\in \Lambda$ the disc
$W^s_x=W^s(p)=\{(x',y',z'):x'=x\}$ is a (principal) component (in $M$) of the stable manifold of $p$ and the interval  $W^{ss}_{x,y}=W^{ss}(p)=\{(x',y',z'):x'=x, y'=y\}$ is a (principal) component of strong stable manifold of $p$. Unstable manifolds $W^u(p)$ are more complicated, each is dense in $\Lambda$ and for each $x,x'\in \R^1$ the unstable lamination of $\La$ defines the holonomy map $h_{x,x'}:W^s_{x/2\pi\Z}\cap \La\to W^s_{x'/2\pi\Z}\cap\La$.

Sometimes we write $W^s_x$ in place of $W^s_{x/2\pi\Z}$. Denote by $\pi_x$ the projection $(x,y,z)\mapsto x$.
The part of global $W^u(p)$, which is the lift of $[x,x']\subset\R$ for
$\pi_x$ will be denoted $W^u_{[x,x']}(p)$.
For $[x,x']$ equal to $[0,2\pi]$ or slightly bigger, clear from the context, we shall write sometimes just $W^u(p)$.

\medskip

Denote by $\pi_{x,y}$ the projection $(x,y,z) \mapsto (x,y)$. We assume in this paper the following
\un{\emph{transversality assumption}}:  each intersection of two distinct $\pi_{x,y}(W^u(p))$ and
$\pi_{x,y}(W^u(q))$ is transversal.

Let $\mu=\mu_{t_0}$ be the Gibbs measure (equilibrium state) on $\Lambda$ for the potential $t_0\log |\la'|$,
where $t=t_0$ is zero of the topological pressure $t\mapsto P(f,t\log \la')$. The measure $\mu$
can be called \emph{geometric} or SRB in stable direction or just stable SRB-measure. Denote by $\mu^s_x$ its conditional measures on
$W^s_x$ for each $x$, see explanations following Lemma \ref{Lip_full_meas}.

\medskip

\begin{definition*}[Thin solenoid]
The solenoid $\La$ for injective $f:M\to M$ as in \eqref{triangular}
satisfying $\chi_{\mu}(\nu')<\chi_{\mu} (\lambda')< -\chi_{\mu}(\eta')$ for
$\mu$ being the stable SRB-measure on $\La$, for Lyapunov exponents $\chi_{\mu}(\xi):=\int \log\xi d\mu$ for $\xi=\nu',\lambda'$ and $\eta'$ respectively,
is  called
a \emph{non-uniformly thin}, or just \emph{thin}, solenoid.
\end{definition*}

\medskip

We prove the following

\begin{thmA}\label{thA}
 Let $\La$ be a non-uniformly thin solenoid for $f:M\to M$ as in the definition above, which satisfies the transversality assumption.
Then, for $\HD$ denoting Hausdorff dimension and for every $x\in S^1$,

1. $\HD(\Lambda\cap W^s_x)=t_0$.

2. $\HD(\Lambda)=1+t_0$.

\end{thmA}

\begin{thmB} Under the assumptions of Theorem A,
 for $\Pi_t$ denoting packing measure in dimension $t$, for every $x\in S^1$ it holds

$0<\Pi_{t_0}(\Lambda \cap W^s_x)<\infty$.
Moreover $ 0<\Pi_{1+t_0}(\Lambda)<\infty$.

\end{thmB}

 \begin{thmC} Under the assumptions of Theorem A, for ${\rm HM}_t$ denoting Hausdorff measure in dimension $t$, ${\rm HM}_{t_0}(\Lambda \cap W^s_x)=0$ for every $x\in S^1$. Moreover  ${\rm HM}_{1+t_0}(\Lambda) =0$.
\end{thmC}

Now, we need the following
\begin{definition*}[Bunching condition]
We say that our thin solenoid $\Lambda$ satisfies the bunching condition if
$\eta'(x) > \lambda'(x)/\nu'(x)$ for every $x\in\Lambda$.
\end{definition*}

A theorem used in particular to compare sizes of $\La\cap W^s_x$ for varying $x$,  is
\begin{thmD} Under the assumptions of Theorem A, if the bunching condition is satisfied,  then all the holonomies $h_{x,x'}$ for $x,x'\in\R$
are uniformly Lipschitz continuous.

If the bunching condition is not assumed, then for each $R$ there
exists $Lip (R)>0$
such that for each $x$ there is a set $L_x\subset  W^s_x\cap\La$ such that for all $x'$ satisfying $|x-x'|<R$ the holonomies $h_{x,x'}$ are locally bi-Lipschitz continuous with a common constant $Lip(R)$
and

\noindent $\mu^s_x(\La \cap W^s_x  \setminus L_x)= 0$.

In fact $\mu^s_x(NL^w \cap W^s_x)=0$ for certain weak non-Lipshitz set $NL^w$ invariant under all $h_{x,x'}$
for $0\le x,x'\le 2\pi$, which intersected with $W_x^s$ is bigger than the complement of Lipschitz $L_x$.
Moreover $\HD (NL^w \cap W^s_x)  < t_0$.

\end{thmD}

 Here "local" means for every $p\in L_x$ there exists $\delta$ such that for all $q\in W^s_x\cap \La \cap B(p,\delta)$ and $|x'-x|<R$  the Lipschitz condition with the constant $Lip(R)$ holds for $h_{x,x'}$.

 \medskip

 \emph{Bunching} condition above, appeared in a related setting in \cite {PSW}, see also \cite[Th. 4.21]{CP}
 with a stronger conclusion that the unstable foliation is $C^1$.

 \medskip

Some of the assertions above hold also for the projections to the $\{(x,y)\}$ plane, in particular
\begin{thmE}
Under the assumptions of Theorem A

\noindent$\HD(\pi_{x,y}(\Lambda\cap W^s_x))=t_0$ and
$\HD(\pi_{x,y}(\Lambda))=1+t_0$.
\end{thmE}
We do not know if $0<\Pi_{t_0}(\pi_{x,y}(\Lambda \cap W^s_x))<\infty$ or
$ 0<\Pi_{1+t_0}(\pi_{x,y}(\Lambda))<\infty$.

\medskip

The assertions of Theorem A almost automatically hold for Hausdorff dimension replaced by upper box dimension $\ov{\BD}$.
Indeed, the estimates from below follow from $\HD\le\ov{\BD}$. The estimate $\ov{\BD}(\Lambda\cap W^s_x)\le t_0$
follows from Lemma \ref{upper}. The estimate $\ov{\BD}(\Lambda)\le 1+ t_0$  follows from Proof of Theorem \ref{Affine Main Theorem}, Step 3.

\bigskip

{\bf On the linear case.} The mapping $f$ in \eqref{triangular} is called lower triangular (because such is the differential $Df$ in the $y,z$ direction) non-linear. Our paper complements the study of the linear diagonal case
\begin{equation}\label{linear}
f(x,y,z)=(dx({\rm mod} 2\pi), \lambda y  + u(x), \nu z  + v(x) ),
\end{equation}
with $0<\nu<\lambda< 1/d$.
It was done by B. Hasselblatt and J. Schmeling in \cite{HS}, where nevertheless there were hints concerning the non-linear situations, and by M. Rams and K. Simon \cite{RS}. Namely, Theorems A,B,D generalize \cite{HS} and  Theorem C generalizes \cite{RS}.
By the way Theorem B was proved in \cite{RS} only for Lebesgue almost all $x$.

\smallskip

{\bf Transversally conformal case.} This is the case where $f$ is conformal on every $W^s$, well understood.
Theorems A and B hold, though the dimension $t_0$ can be larger than 1 (in a thick case). Packing and Hausdorff measures on $W^s$ in dimension $t_0$ are equivalent. In fact this is transversally complex 1D situation, whereas our non-conformal case corresponds after $\pi_{x,y}$-projection to transversally real 1D situation with overlaps.

\bigskip

{\bf Motivation.}
Let us present the geometric picture of the solenoid. It helps to think not of the solenoid itself (which is locally just a Cantor bouquet of almost parallel lines, hard to analyze with an untrained eye) but of its approximations $f^n(M)$. Each of those is a tube, winding around along $S^1$, going around $d^n$ times. Thus, every section of $f^n(M)$ with a disc $W_x^s=\{x\}\times \mathbb D$ is a disjoint union of $d^n$ components, each of those being a (slightly deformed) ellipse, with exponentially increasing ratio of the large semiaxis to the small semiaxis, and all the large semiaxes pointing roughly in the same $\{y\}$ direction. See Figure 1. As already mentioned, those aproximate ellipses are disjoint, but there are plenty of sections in which some of the ellipses are very close to each other, in a distance exponentially smaller than their diameters. One of the main concerns in our study will be the understanding of the way those ellipses 'move' as we move the section plane around $S^1$.

This picture resembles very much the picture of an affine iterated function system, or maybe even better: the averaged picture of many different but similar affine iterated function systems (as each ellipse in the section comes with different backward path, and is thus produced by a different collection of nonconformal contracting maps). An important element of the picture is that those contracting maps in the sections satisfy a version of the domination property, that is the strong contracting direction in one iteration stays close to the strong contracting direction in the following iterations. That is, we clearly have two different negative Lyapunov exponents in our system.

This picture lets us expect the behaviour similar to the known generic behaviour of affine iterated function systems. We expect the Hausdorff and packing and box dimensions of each section to be given by the Falconer's singular value pressure formula, which in our 'thin' case should depend only on what happens in the expanding and weakly contracting directions (in particular, the dimensions should be preserved by the projection to the $\{(x,y)\}$ plane). Moreover, like in simpler solenoid cases, we expect that the situations when two ellipses pass nearby (which are unavoidable by the very geometry of the solenoid) should have an effect strong enough to zero the Hausdorff measure, but the packing measure should stay positive and finite. And those are exactly the statements we eventually prove in Theorems A-C and E.




\bigskip

Hasselblatt and Schmeling stated in \cite{HS}\footnote{See some history and other references there, in particular \cite{Bothe}.} the following

\begin{con} The fractal dimension of a hyperbolic set is the sum of those of its stable and
unstable slices, where "fractal" can mean either Hausdorff or upper box dimension.
\end{con}
For solenoids, in
\cite{HS} and here, an affirmative answer on Hausdorff dimension has been proven. Hausdorff dimension in the stable direction is $t_0$ and in unstable 1, that is $1+t_0$  together. Notice that this is dimension $t_0$
of conditional measures of $\mu$ geometric (SRB) in stable direction (see above) and of dimension 1 of an SRB measure in the unstable direction. Both SRB measures are usually different (even mutually singular), unless (e.g.) in  the diagonal linear case,
where both measures coincide with the measure of maximal entropy.

For any invariant hyperbolic measure $\nu$ indeed $\HD(\nu)=\HD(\nu^s)+\HD(\nu^u)$, see \cite{BPS}, but even supremum of $\HD(\nu)$ over invariant $\nu$ on $\Lambda$ can be less than $\HD(\Lambda)$. See e.g. \cite{Rams}. So in a general case
one is forced to use the both SRB measures.

\smallskip

Finally note that Hasselblatt and Schmeling relax the assumption of transversality to the  assumption the intersections of $\pi_{x,y}$ projections of $W^u$'s are non-flat. This in particular holds for all real analytic (that is with the functions $u,v$ real analytic) linear solenoids, see \cite{HS}.
A natural challenge would be to generalize our non-linear theory to a general real-analytic (non-transversal) case.

\bigskip

{\bf Outline.} In Section 3 we prove a part of Theorem D saying that the points in $W^s_x$ where a holonomy $h_{x,x'}$ is not locally bi-Lipschitz has measure $\mu$ equal 0. This follows (and clarifies) \cite{HS}.
In fact we prove that a bigger set has measure 0, the set of $p$ which are not \emph{strong Lipschitz}, called above \emph{weak non-Lipschitz}. For such a $p$ the projection  $\pi_{x,y}(W^u(p))$
intersects  some  $\pi_{x,y}(W^u(q))$'s for $q\notin W^u(p)$ arbitrarily close to $W^u(p)$. Equivalently $p$ is
\emph{strong Lipschitz} if
$W^{ss}_{\rm loc}(p)=\{p\}$,
hence counting Hausdorff dimension only $E^s/E^{ss}$ counts so $\HD(\Lambda\cap W^s)=h_\mu(f)/-\chi_\mu(\la') = t_0$, where $h_\mu(f)$ is the measure (Kolmogorov's) entropy, \cite{LY}.
This is done in Section 4, and yields Theorem A. Again we roughly follow \cite{HS}.

In Section 3,
Theorem D is in fact
proved under  the assumption stronger than $ \chi_\mu(\la')< - \chi_\mu(\eta') $, namely under the assumption
$\sup\la'<1/\sup \eta'$, called
\emph{uniform dissipation}.

Note that Lipschitz property is related with Theorem A on Hausdorff dimension a little bit by chance, saying however that $\HD(W^s_x\cap\La)$ does not depend on $x$. In fact holonomy being Lipschitz is a weak condition, e.g. it holds for all holonomies $h_{x,x'}$
provided  $\eta'>\la'/\nu'$ as in Theorem D (well known), as twisting, hurting Lipschitz property, cannot develop if $f^{-n}$ squeezes (by $(\eta')^{-1}$) too much.  Lipschitz property is crucial to conclude
$\HD(W^s_x\cap\La)=t_0 \Rightarrow \HD(\La)=1+t_0$ in Theorem A. Compare Conjecture above.

Theorem B is proved in Section 5. The proof has common points with \cite{RS}.
Analysis is more delicate than in the proof of Theorem A. We prove that for $\mu$-a.e. $p$ for a sequence
of $n$'s the $\pi_{x,y}$-projection of the tube $f^n(M)$ (truncated to $[0,2\pi]$), called of order $n$  containing $p$, intersects only a bounded number of other projections of tubes of order $n$.

Theorem C is proved in Section 6, again using an idea from \cite{RS}. It uses the fact of arbitrarily high
multiplicity of overlapping of projections of tubes of order $n$ for $\mu$-a.e. $p$.

The estimate $\HD(NL^w_x)<t_0$ in Theorem D is proved in Section 7, together with a more precise estimate,
following from large deviations estimate concerning Birkhoff averages.

Theorem E follows from other theorems because the assertions on dimensions are verified on the sets where the projection $\pi_{x,y}$ is finite-to-one.

Section 7 contains also a remark on general Williams 1-dimensional expanding attractors and a remark on a
possibility of integrating general solenoids to triangular as in \eqref{triangular} ones.

\

{\bf Acknowledgements.}
We wish to thank Adam Abrams for making pictures to this paper.
We are grateful to Aaron Brown and J\"org Schmeling for useful discussions.
All the authors are partially supported by Polish NCN grant
2019/33/B/ST1/00275.

\bigskip

\section{Holonomy  along unstable lamination }

\begin{definition}\label{horizontal and vertical}
We will now introduce a symbolic description on the attractor, defining an `almost bijection'
$\rho: \Sigma_d \to \Lambda$, where $\Sigma_d = \{0,1,\ldots, d-1\}^\Z$ is the usual two-sided full shift space on $d$ symbols.
Why it is only an `almost bijection' will be explained soon.
Moreover, it will be done in such a way that
$\rho$ semi-conjugates  the left shift $\varsigma$ acting on $\Sigma_d$ to $f|_\Lambda$, namely
$\rho \circ \varsigma = f \circ \rho$.
Denote else the elements of $\Sigma_d$ by

\noindent $\underline{i}=(\ldots, i_{-n}, \ldots, i_0| i_1, \ldots, i_n, \ldots)$, where the vertical line separates entries with non-positive indices from the entries with positive indices.

Let us start closer explanations from the $x$ coordinate. Looking at the formula \eqref{triangular} we see that if $(x',y',z')=f(x,y,z)$ then $x'$ does depend only on $x$, not on $y$ nor $z$. The restriction of $f$ to the first coordinate is the $d$-to-1 expanding map $\eta$.
Denote by $a_0,...,a_{d-1},a_d=a_0$ the points of $\eta^{-1}(0)\in \R/2\pi\Z=S^1$ numbered in the increasing  order.
We can assume that $a_0=0=\eta(0)$.
For $i=0,1,\ldots, d-1$ we denote $V_{|i}:=[a_i, a_{i+1}]\times \mathbb D$; those sets will be called  \emph{vertical cylinders of level 1}. We can then define for every $n=1,2, \ldots$ the vertical cylinders of level $n$, by
$$
V_{i_1,\ldots,i_n}=\{p=(x,y,z)\in M:\ f^{k-1}(p)\in V_{|i_k} {\rm for}\ k=1,\ldots,n\}.
 $$
 For every sequence $(i_1, i_2,\ldots )\in \{0,\ldots, d-1\}^\N$ there exists exactly one $x\in S^1$ such that
\begin{equation} \label{eqn:xsymb}
\bigcap_{n=1}^\infty V_{|i_1,\ldots,i_n} = \{x\} \times \mathbb D,
\end{equation}
and vice versa: for all except countably many points $x\in S^1$ there exists exactly one sequence
$(i_1, i_2,\ldots )\in \{0,\ldots, d-1\}^\N$ such that \eqref{eqn:xsymb} holds. The exceptions are the points $x$ such that $\eta^k(x)=0$ for some $k\in\N$. For each of those points one can find exactly two sequences satisfying \eqref{eqn:xsymb}. We note that $f^{k-1}(p)\in V_{|i_k}$ is equivalent to $\eta^{k-1}(x)\in [a_{i_k},a_{i_k+1}]$, so what we described up to this point is the usual construction of symbolic description for an expanding map of the circle.


Let us now define the horizontal cylinders of level $n=0,1,2, \ldots $ by the formula:

\[
H_{i_{-n},\ldots, i_0|} := f^{n+1}(V_{|i_{-n},\ldots, i_0}),
\]
and then define

\[
\rho(\underline{i}):= \lim_{n\to\infty} V_{|i_1,\ldots, i_n} \cap H_{i_{-n},\ldots, i_0|}.
\]
The fact that $\rho$ semi-conjugates $\varsigma$ to $f|_\Lambda$ is clear from the definitions.

\smallskip

We will denote by $V(n)$ and $H(n)$ the sets of all vertical (resp. horizontal) cylinders of level $n$. For a given $p\in\Lambda$ we will  denote by $V_n(p)$ and $H_n(p)$ the vertical and horizontal cylinder of  $n$, containing $p$.
Sometimes we will just write $H_n$ and $V_n$ if we do not specify $p$.

\end{definition}

We note that $\rho(\underline{i})=(x,y,z)$, with $x$ depending on $i_1, i_2, \ldots$ and $(y,z)$ depends on $x$ and on $i_0,i_{-1},\ldots$. It makes thus sense to write $\Sigma_d = \Sigma_d^- \times \Sigma_d^+$ with $\Sigma_d^-, \Sigma_d^+$ denoting the one-sided shift spaces on $d$ symbols, the former one given by nonpositive entries and the latter one by the positive entries. We then denote
\[
W^u(\underline{i}) = \bigcap_{n=0}^\infty H_{i_{-n},\ldots, i_0|},\ \ \
W^s(\underline{i}) = \bigcap_{n=1}^\infty V_{|i_1,\ldots, i_n}.
\]
Clearly, $\rho(\underline{i}) = W^u(\underline{i}) \cap W^s(\underline{i})$. We will also use the notation $W^u(p), W^s(p)$ for $p\in\Lambda$, as the shortcut for $W^u(\rho^{-1}(p)), W^s(\rho^{-1}(p))$. We note that as $\eta$ is expanding and $\lambda, \nu$ are contracting, $W^u(p)$ and $W^s(p)$ are pieces of the unstable and stable manifolds at $p\in\Lambda$ -- which explains the notation used. This notation has been already introduced  in Introduction, together with the notation $W^s_x$ for $x\in S^1$.





\smallskip



\medskip

\begin{definition}\label{Gamma}

Denote  $\hat f:= \pi_{x,y}\circ f \circ  (\pi_{x,y})^{-1}$, where $\pi_{x,y}$ is the projection $(x,y,z)\mapsto (x,y )$, see Introduction. This definition makes sense, since $f$ preserves the foliation $\{W^{ss}_x\}$ (vertical intervals here).

To simplify notation denote sometimes objects being the projection of objects in $M$ by $\pi_{x,y}$ by adding
a hat over them, e.g.  $\hat\Lambda:=\pi_{x,y}(\Lambda)$ or $\hat p =\pi_{x,y}(p)$.

Remind that the projection $(x,y,z)\mapsto x$ is denoted by $\pi_x$, see Introduction. For any $p\in M$ the point $\pi_x(p)$ will be sometimes denoted by $x(p)$.

\bigskip

Denote
$
\Gamma:=\{\hat p \in\hat\Lambda: \exists (...,i_0|)$ and  $\exists (...,i'_0|)$ with $i_0\not=i_0'$,
such that
 $\hat W^u(\rho(...,i_0|))$ and
$W^u(\rho(...,i_0'|))$ intersect at $\hat p$. Here $W^u=W^u_{[-L\eta_n^{-1}, 2\pi + L\eta_n^{-1}]}$,
where the `margins' $L\eta_n^{-1}$ will be defined in Notation \ref{chain} and Definition \ref{def_strong_Lip}.

\smallskip

In words, $\Gamma$ is a  Cantor set, consisting of the intersections of the $\pi_{x,y}$-projections  of
$W^u$'s belonging to  different (slightly extended) horizontal cylinders of level 0. It discounts intersections of the projections of $W^u$'s being in the same cylinder of level 0. See Figure 3.

\end{definition}

\

{\bf (Unstable) transversality assumption.}\ All intersections of these lines (i.e. projections by $\pi_{x,y}$ of  unstable manifolds) with different $i_0$'s are in this paper assumed to be mutually transversal.

\begin{remark}\label{transversality}
Notice that by the compactness argument and the continuity of the sub-bundle $E^u_{\Lambda}$, by the transversality assumption all the
intersection angles are bounded away from 0, say by $\alpha_0$.

Also by compactness and continuity of $E^u$ on $\Lambda$ there exists $r_0>0$ such that if for $p,p' \in\Lambda\cap W^s$ their mutual Euclidean distance is $r<r_0$ and their $i_0$'s are different then the distance of their $\pi_{x,y}$ projections from $\Gamma$, more precisely from the intersection $\hat W^u(p)\cap \hat W^u(p')$ which is in particular nonempty,  is bounded by
$2r / \tan \alpha_0$.
\end{remark}

\medskip

Remember that we consider $f$ in the form \eqref{triangular} and write
 $\eta':=\frac{\partial \eta}{\partial x}$, $\lambda':=\frac{\partial \lambda}{\partial y}$ and $\nu':=\frac{\partial \nu}{\partial z}$. Then

\begin{notation}\label{chain}
We write $\xi_n^+(p)=\xi(p)\xi(f(p))...\xi(f^n(p))$ for $\xi=\eta', \lambda'$ or $\nu'$ respectively.
Write also $\xi_n^-(p):= \xi_n^+(f^{-n}(p))$.

\end{notation}

\begin{definition}\label{def_strong_Lip}
A point $p\in\Lambda$ is said to be {\it strong locally Lipschitz}
 if there is $L>0$ such that 
for all $n$ big enough, denoting $(\eta_n^-(p)$ by $\eta_n$,

\begin{equation}\label{SL}
\dist(\widehat{V_n(f^{-n}(p))},  \Gamma\cap \widehat{W}^u_{[-L\eta_n^{-1}, 2\pi + L\eta_n^{-1}]}(f^{-n}(p))
\ge L\eta_n^{-1},
\end{equation}
with the distance in ${W}^u$  measured between the projections by $\pi_x$ in $\R$.

Equivalently we could replace here $\widehat{V_n(f^{-n}(p))}$ by $\widehat{f^{-n}(p)}$. It would influence the constant $L$ only.

\smallskip

By the unstable transversality and transversality of intersection of stable and unstable foliations, this is equivalent to the distance in the $\{(x,y)\}$-plane satisfying
\begin{equation}\label{SL22}
\dist(\widehat{f^{-n}(p)}, \widehat{W}^u(p') \ge  \Const L \eta_n^{-1}
\end{equation}
for all $p'$ having $i_0$ different from the $i_0$ for $f^{-n}(p)$.

\smallskip

We call all points $p$ which are strong locally Lipschitz with the constant $L$ such that \eqref{SL} holds for all $q\in W^u(p)$ in place of $p$  {\it strong locally bi-Lipschitz}.

\medskip

Notice that this definition allows to say that the whole $W^u(p)$ is strong locally bi-Lipschitz and write

\begin{equation}\label{SBL}
\dist(\widehat{f^{-n}(W^u(p))}, \Gamma\cap \widehat{W}^u_{[-L\eta_n^{-1}, 2\pi + L\eta_n^{-1}]} (f^{-n}(p))
\ge L\eta_n^{-1}.
\end{equation}

\end{definition}

\begin{remark}\label{bi}
Notice that if for $\hat L>0$ strong locally Lipschitz condition
$\dist(\widehat{f^{-n}(p)}, \Gamma\cap \hat W^u(f^{-n}(p)))\ge \hat L
\eta_n (f^{-n}(p))^{-1}$ holds and $q\in W^u_{[0,2\pi]}(p)$ then
$\dist(\widehat{f^{-n}(q)}, \Gamma)\ge (\hat L - \Const)
\eta_n (f^{-n}(q))^{-1}$. So for \eqref{SL} satisfied at $p$ with $\hat L > 2\Const$ strong locally Lipschitz condition holds for all $q\in W^u(p) $, with $L=\hat L /2$. So $p$ is strong locally bi-Lipschitz.
\end{remark}

\begin{definition}\label{holonomy}
For every $p\in \Lambda$ and $q\in W^u(p)\setminus \{p\}$ one defines the holonomy map
 $h_{x(p),x(q)}: W^s(p)\to W^s(q)$ along unstable foliation (lamination) $\cW^u$ by $h_{x(p),x(q)}(v)$
 being the only intersection point of $W^u(v)$ with $W^s(q)$.
 \end{definition}

\begin{theorem}\label{strong_Lip}
For every $L_2>0$ there exists $L_1 >0$ such that
for each $p$ strong locally (bi)Lipschitz with the constant $L=L_1$ there exists $n(p)$
such that for each $q\in W^u_{[0,2\pi]} (p)$
the holonomy between
$W^s(p)\cap\Lambda$ and $W^s(q)\cap\Lambda$ ,
in $H_{n(p)}(p)\cap \Lambda$, is locally bi-Lipschitz
continuous at $p$ with Lipschitz constant $L_2$.
\end{theorem}
Here \emph{at $p$}  means that for every $p'\in W^s(p)\cap H_{n(p)}(p)\cap\Lambda$
we have
$$
L_2^{-1}\dist (p,p') \le  \dist (h_{\pi_x(p),\pi_x(q)}(p), h_{\pi_x(p),\pi_x(q)}(p') \le L_2 \dist (p,p'),
$$
where dist is  the euclidean distance in $\D$.

\begin{proof} We repeat (adjust) the calculations in \cite{HS}.
Consider
$q\in W^u(p)$ and $p'\in W^s(p)\cap\Lambda$.
Define $q':= h_{x(p),x(q)} (p')$.
Let $p'\in H_n(p)\setminus H_{n+1}(p)$, that is $p_{-n}=f^{-n}(p)$ and $p'_{-n}=f^{-n}(p')$ are in different $H_0$'s.

Local Lipschitz continuity of the holonomy $h_{x(p),x(q)}$ at $p$ would follow from the
existence of a uniform upper bound of
\begin{equation}\label{ratio}
\dist(q,q') / \dist(p,p')
\end{equation}
 for $p'$ close enough to $p$, i.e. $n$ defined above, large.

It is comfortable to consider the distance
$d=d_1 + d_2$, the distances in the $y$ and $z$ coordinates.

We shall use the triangular form of the differential
$
Df|_{\{y,z\}} =
\begin{bmatrix}
\lambda' & 0 \\  a & \nu'
\end{bmatrix}.
$
Due to $\nu'<\lambda'$ we have $Df^n|_{\{y,z\}}=\begin{bmatrix}
\lambda_n' & 0 \\  a_n & \nu_n'
\end{bmatrix}$ where $|a_n|\le \Const \lambda_n$.
Write, according to the decomposition $d=d_1+d_2$,

\noindent $
d(f^{-n}(p),f^{-n}(p')):=\ov\Delta p=\ov\Delta_1 p + \ov\Delta_2 p$ \; \; {\rm and}

\noindent $d(f^{-n}(q),f^{-n}(q')):=\ov\Delta q=
\ov\Delta_1 q  + \ov\Delta_2 q.$

We estimate
\begin{align*} &
d(q,q') =  \ov\lambda_n (f^{-n}(q))\ov\Delta_1q  +  |\ov{a}_n(f^{-n}(q))\ov\Delta_1 q +
\ov\nu_n(f^{-n}(q))\ov\Delta_2 q| \le  \\  &
\ov\lambda_n (f^{-n}(p))\ov\Delta_1p+|\ov{a}_n(f^{-n}(p)))\ov\Delta_1 p  +
\ov\nu_n(f^{-n}(p))\ov\Delta_2 p|  +  \\ &
\ov\lambda_n (f^{-n}(p)) A/\eta_n(f^{-n}(p))
\end{align*}
for a constant $A$ depending on the angle between $W^u$ and $W^s$.
Here $\ov\lambda_n, \ov{a_n}, \ov\nu_n$ are averages of derivatives $\lambda_n,a_n,\nu_n$ respectively,
on appropriate intervals, namely
integrals divided by the lengths of the intervals, horizontal along $y$ for two first integrals and vertical along $z$ for the last one.

\medskip

On the other hand
$$
d(p,p') = \ov\lambda_n (f^{-n}(p)) \ov\Delta_1p + |\ov{a}_n(f^{-n}(p))\ov\Delta_1 p +
\ov\nu_n(f^{-n}(p))\ov\Delta_2 p|
$$

To obtain an upper bound of \eqref{ratio} it is sufficient to assume the existence of an upper bound of the  ratio of the above quantities, namely (omitting $(f^{-n}(p))$ to simplify notation) 

\begin{center}
\begin{tikzpicture}[scale=2]
	\def\preP{(0.4,1.0)}
	\def\midP{(0.4,1.9)} 
	\def\prePprime{(0.6,2.1)}
    \draw [gray,bend left=5] (0.2,2.1-0.3162) to +(1.0,0.05);
    \draw [gray,bend left=5] (0.4,2.1) to +(1.4,0); 
    \draw [gray,bend left=5] (0.59,2.1+0.3162) to +(1.075,0.05);
    \draw [gray,bend left=5] (0,1-0.3162) to +(1.2,0.05);
    \draw [gray,bend left=5] (0.2,1.0) to +(1.4,0); 
    \draw [gray,bend left=5] (0.39,1+0.3162) to +(0.81,0.05);
	\begin{scope}[xshift=-0.2cm]
	    \draw (0,0) -- (0,2) -- (1,3) -- (1,1) -- cycle;
	    \fill \prePprime circle (1pt) node [right=1.5em,above] {$f^{-n}(p')$};
	    \draw [gray, rotate around={45:\prePprime}] \prePprime circle (0.4 and 0.2);
	    \fill \preP circle (1pt) node [below] {$f^{-n}(p)$};
	    \draw [gray, rotate around={45:\preP}] \preP circle (0.4 and 0.2);
	    \draw [dashed] \preP -- node [midway,left] {\scriptsize$\ov\Delta_2p$} \midP -- node [midway,above=-0.3em] {{\scriptsize$\ov\Delta_1p$}\hspace*{1.5em}} \prePprime;
    \end{scope}
	\begin{scope}[xshift=1.2cm]
	    \draw [fill=white,fill opacity=0.67] (0,0) -- (0,2) -- (1,3) -- (1,1) -- cycle;
	    \fill \preP circle (1pt) node [below] {$f^{-n}(q)$};
	    \fill \prePprime circle (1pt) node [right] {$f^{-n}(q')$};
	    \draw [dashed] \preP -- node [midway,left] {\scriptsize$\ov\Delta_2q$} \midP -- node [midway,above=-0.3em] {{\scriptsize$\ov\Delta_1q$}\hspace*{1.5em}} \prePprime;
    \end{scope}

    \def\P{(0.45,0.9)} \def\Pprime{(0.5,1.3)}
    \def\Q{(0.3,0.8)} \def\Qprime{(0.7,1.3)}
    \begin{scope}[yshift=-3.5cm,xshift=-1.3cm]
    \draw plot [smooth] coordinates {\P (1.4,1.1) (2.7,0.8) (3.3,0.8)};
    \draw plot [smooth] coordinates {\Pprime (1.4,1.4) (2.7,1.15) (3.7,1.3)};
    \draw (2.3,1.2) node [above] {$W^u(p')$} (2.1,0.92) node [below] {$W^u(p)$};
	\begin{scope}
	    \draw (0,0) -- (0,2) -- (1,3) -- (1,1) -- cycle;
	    \fill \Pprime circle (1pt) node [above=0.4em] {$p'$};
	    \draw [gray, rotate around={45:\Pprime}] \Pprime circle (0.2 and 0.1);
	    \fill \P circle (1pt) node [below=0.5em] {$p$};
	    \draw [gray, rotate around={45:\P}] \P circle (0.2 and 0.1);
    \end{scope}
	\begin{scope}[xshift=3cm]
	    \draw [fill=white,fill opacity=0.75] (0,0) -- (0,2) -- (1,3) -- (1,1) -- cycle;
	    \fill \Qprime circle (1pt) node [below=0.4em] {$q'$};
	    \draw [gray, rotate around={45:\Qprime}] \Qprime circle (0.2 and 0.1);
	    \fill \Q circle (1pt) node [below=0.5em] {$q$};
	    \draw [gray, rotate around={45:\Q}] \Q circle (0.2 and 0.1);
    \end{scope}
    \end{scope}
    \draw [<-, bend right=10] (-0.4,0.2) to node [midway,left] {$f^{-n}$\!} (-0.8,-0.7);
    \draw [<-, bend left=10] (1.8,0.2) to node [midway,right] {$f^{-n}$} (2.2,-0.7);
\end{tikzpicture}
\end{center}

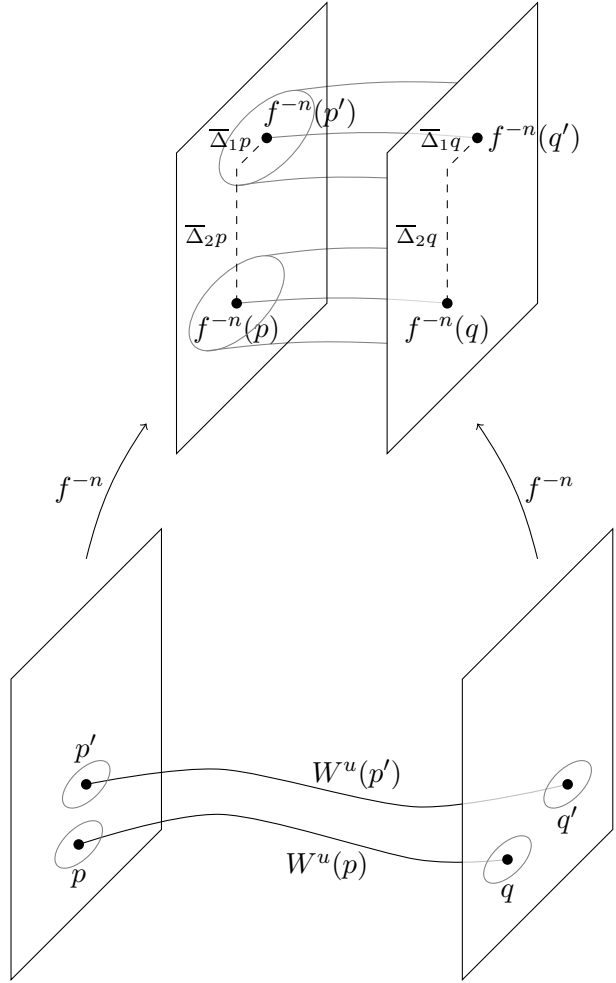
FIGURE 2. Holonomy twist.

\bigskip

$$
1 + \frac
{A\lambda_n  / \eta_n}
{\lambda_n \ov\Delta_1 p +
|a_n\ov\Delta_1p + \nu_n\ov\Delta_2 p | }.
$$
We needed bars over $\la,\nu,a$ to reduce above a fraction to the summand 1.
From now on these bars (integrals) are not needed.

We conclude calculations with sufficiency to assume the existence of an upper bound of

\begin{equation}
\frac1{ (\ov\Delta_1p + \frac{|a_n\ov\Delta_1p + \nu_n\ov\Delta_2p|}{\lambda_n} )\eta_n}
\end{equation}

or to assume that
the inverse
$$
(\ov\Delta_1p + \frac{|a_n\ov\Delta_1p + \nu_n\ov\Delta_2p|}{\lambda_n} )\eta_n
$$
is bounded away from 0.

\medskip

Thus, Lipschitz property follows  from either of
\begin{equation}\label{cone_condition}
\ov\Delta_1 p \; \eta_n \ge \Const > 0
\end{equation}
or
\begin{equation}\label{spec_gap}
 \frac{|a_n\ov\Delta_1p + \nu_n\ov\Delta_2p|}{\lambda_n} \eta_n \ge \Const > 0.
\end{equation}

The condition  \eqref{spec_gap}, in the diagonal case $a_n=0$, means that
the contraction in the space of stable leaves $W^s$ by $f^{-n}$, along the coordinate $x$,   due to small  $\eta_n^{-1}$ is strong
 enough  to bound
the twisting effect caused by $\log \nu_n /\log\lambda_n$, hence implying the Lipschitz continuity of all the holonomies at $p$ along unstable foliation of a bounded length leaves (e.g. by $2\pi$). This is for $\ov\Delta_1(p)\approx 0$ (hence $\ov\Delta_2(p)$ large). Otherwise Lipschitz condition holds automatically.

The condition  \eqref{cone_condition} is
 equivalent to strong locally Lipschitz \eqref{SL} in the Definition \ref{def_strong_Lip}
 by transversality condition, see Remark \ref{transversality} and \eqref{SL} and \eqref{SL22}.
 This implies that the distance between $W^s(f^{-n}(p))$ and $W^s(f^{-n}(q))$ is bounded by $\Const\times \ov\Delta_1(p)$ hence $\ov\Delta_1(q)\le\Const\ov\Delta_1(p)$
hence $d(q,q')\le\Const d(p,p')$ hence just Lipschitz property of $h_{p,q}$ at $p$.

By Remark \ref{bi} for Const above large enough we obtain strong bi-Lipschitz  property.

\end{proof}

We denote the set of all strong locally bi-Lipschitz points in $\Lambda$ by $L^s$ and $L^s\cap W^s(p)$ with $x(p)=x$ by $L^s_x$. Sometimes we write $\Lambda^s(\tau, L,n )$ for specified $n$, see \eqref{SL}.

As we already mentioned in Theorem \ref{strong_Lip}
\begin{lemma}\label{holonomy_Lip}
 $h_{x,x'}(L_x^s)=L^s_{x'}$ for all $x,x'\in S^1$ for the holonomy
 $h_{x,x'}$ along unstable foliation. The holonomy is locally Lipschitz on $L^s$.
\end{lemma}

\begin{notation}\label{NLw} We call the set complementary to $L^{s}$ in $\Lambda$: weak non-Lipschitz, and denote it by $NL^w$.
By Lemma \ref{strong_Lip} this condition is weaker than non-Lipschitz. It includes some Lipschitz (e.g. if
bunching condition holds, see Theorem D).
\end{notation}

\

Later on we shall need the following fact easily following from the definitions

\begin{lemma}\label{Wss}
For every $p\in L^s$  there exists $n$ such that
$$
W^{ss}(p)\cap\Lambda\cap H_n(p)=\{p\}.
$$
\end{lemma}

\begin{proof} Notice that the existence of $p'\in W^{ss}(p)\cap \Lambda\cap (H_n(p)\setminus H_{n+1}(p))$ is equivalent to $\pi_{x,y}(f^{-(n+1)}(p))\in \Gamma$. If it happens for $n$ arbitrarily large, it contradicts $p\in L^s$.
\end{proof}

\

\section{Lipschitz vs geometric measure}\label{Lip vs geo}

\begin{definition}\label{equilibrium}
 Let $t=t_0$ be the only zero of the pressure function

 \noindent $t\mapsto P(f^{-1}, t\log (\lambda' \circ f^{-1}))$.
 Since $\lambda'<1$, this function is strictly decreasing from $+\infty$ to $-\infty$
Denote by $h^*$ the entropy of the equilibrium measure $\mu=\mu_{t_0}$ for the potential
$t_0\log (\lambda \circ f^{-1})$ (called also stable SRB-measure, see Section \ref{Introduction}).

A geometric meaning of this, is that for an arbitrary $W^s$, replacing $\lambda'$
by a function having logarithm cohomologous to $\log\lambda'$ (denote it also by $\lambda'$), not depending on future $(|i_1,...)$, the quantities  $\log \lambda_n (f^{-n}(p))$ for $p=\rho(...i_0|i_1,...)$ are approximately diameters of
$f^n(W^s(f^{-n}(p)))$ provided $\nu' < \lambda'$.  The quantity $t_0$ which would be
Hausdorff and box dimensions in the conformal case, here in the non-conformal case is only the upper bound of the dimensions of $W^s\cap\Lambda$, so-called "affinity dimension", \cite{Falconer}. The aim of this and the next sections is to prove that $t_0$ is in fact the Hausdorff dimension of all $W^s\cap\Lambda$.

\end{definition}

\medskip

We start now with

\begin{definition}\label{top_ent}
For each ${\underline i}=(i_{-n},...,i_0)$ define

\begin{align*}
& h_n({\underline i}):=
\frac1{n}
\log\#
\Big\{
(i_1,...,i_n): \\ &
\hat H_{i_{-n},...,i_0 |}
\cap
B(\hat V_{|i_1,...,i_n},
L_1 {\eta}^{-1}_n( \pi_x \rho({\underline i}))  )
\cap
\bigcup_{i'_n,...,i'_0\not=i_0}\hat H_{i'_{-n},...,i'_0 |}
\not=\emptyset \Big\},
\end{align*}

where $L_1$ is the constant from Theorem~\ref{strong_Lip}. Define also

\begin{equation}
h_n:=\sup h_n({\underline i}) \;\;\;\; {\rm and}\;\;\;\; h:=\limsup h_n.
\end{equation}
\end{definition}

\

\bigskip


\begin{center}
\begin{tikzpicture}[scale=2]
	\def\curveOne{ plot [smooth] coordinates {(0,2.5) (1,2.5) (2,2.2) (3,1.6) (4,0.9) (5,0.5) (5.1,0.7) (5,0.7) (4,1.1) (3,1.8) (2,2.4) (1,2.7) (0,2.7)} }
	\def\curveTwo{ plot [smooth] coordinates {(0,2.0) (1,2.0) (2,1.7) (3,1) (4,0.35) (5,0.1) (5.1,0.3) (5,0.3) (4,0.55) (3,1.2) (2,1.9) (1,2.2) (0,2.2)} }
	\def\curveThree{ plot [smooth] coordinates {(0,0.85) (1,0.8) (3,0.75) (5,1.4) (5.1,1.6) (5,1.6) (3,0.95) (1,1) (0,1.05)} }
	\def\curveFour{ plot [smooth] coordinates {(0,0.25) (1,0.5) (3,1.25) (5,2) (5,2.2) (3,1.45) (1,0.7) (0,0.45)} }
	\begin{scope}
    	\clip \curveOne \curveTwo;
    	\fill [black!33] \curveThree \curveFour; 
	\end{scope}\begin{scope}\clip(0,0)rectangle(5,3);
		\draw \curveOne \curveTwo \curveThree \curveFour;
	\end{scope}
	\foreach \y in {0.1, 1.2, 1.8, 2.9} \fill [shade, left color=black, right color=white, draw=white] (0,\y) rectangle (1.3,\y+0.04);
	\draw (0,2.45) node [left] {${\hat H}'_0$};
	\draw (0,0.65) node [left] {${\hat H}_0$};
	\draw (5,2.1) node [right] {${\hat H}_n$};
	\draw (5,1.5) node [right] {${\hat H}_n$};
	\draw (5,0.2) node [right] {${\hat H}'_n$};
	\draw (5,0.6) node [right] {${\hat H}'_n$};
    \foreach \x in {1,2,3,4} \draw (\x,0) -- (\x,3);
    \foreach \x in {0.5,1.5,2.5,3.5,4.5} \draw (\x,0) node [below] {$V_n$};
\end{tikzpicture}
\end{center}

\bigskip

\

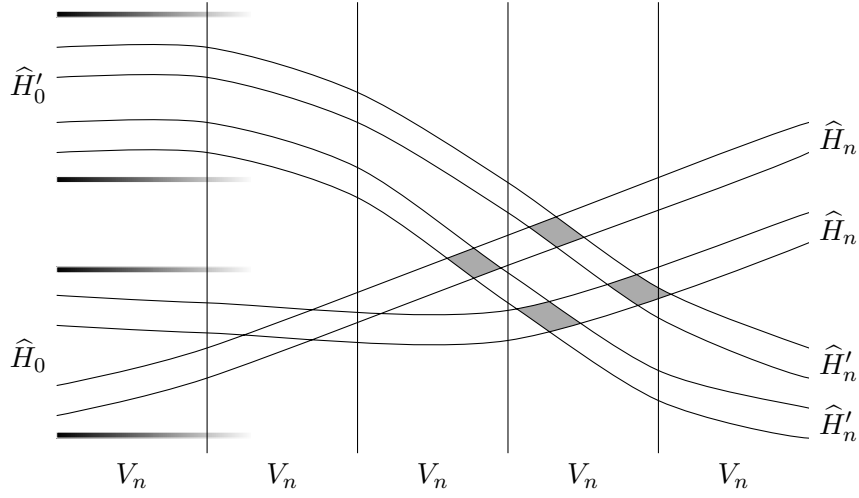
FIGURE 3. Projection to $(x,y)$-plane.
$H_n=H_{i_{-n},...,i_0 |}, H'_n=H_{i'_{-n},...,i'_0 |},$

\noindent $V_n=V_{| i_1,...,i_n}$.

\bigskip

\bigskip

Similarly

\begin{definition}\label{top_ent_infty}
For infinite $\underline i=(...,i_{-n},...,i_0|)$ and $H_{\un{i} }=W^u(p)$ for  $p\in \rho(\un{i})$

\begin{align*}
& h^\infty_n({\underline i}):=
\frac1{n}\log\#\Big\{(i_1,...,i_n):  \\ &
\hat H_{\un{i}}\cap \hat V_{| i_1,...,i_n} \cap B(\Gamma\cap \hat H_{\un{i}},
L_1 {\eta}^{-1}_n( \pi_x \rho({\underline i}))  )
\not=\emptyset\Big\},
\end{align*}

compare \eqref{SL},
and

\begin{equation}\label{uniform}
h^\infty_n:=\sup h^\infty_n({\underline i}) \;\;\;\; {\rm and}\;\;\;\; h^\infty:=\limsup h^\infty_n.
\end{equation}

\end{definition}

\bigskip

The following follows easily from the definitions and the transversality assumption

\begin{proposition}\label{hh}
$h^\infty$ and $h$ are independent of $L_1$ large enough.
Moreover $h^\infty\le h$.
The opposite inequality also holds if $\sup \lambda'< 1/ \sup\eta'$ and $\nu'(p)<\lambda'(p)$ for every $p\in\Lambda$, the property we name: \un{uniform dissipation}.
\end{proposition}

\medskip

\begin{lemma}\label{h<h*}
Assume transversality and 
$\chi_{\mu_{t_0}}(\lambda'))< -\log \sup\eta'$ (call it \un{half-uni\-form dissipation}).
Then
\begin{equation}\label{hinfty<h*}
h^\infty <h^ *.
\end{equation}
\end{lemma}

\begin{proof}

For an arbitrary $\e>0$ and $n$ large enough, denoting by $\ov{\BD}$ the upper box dimension, we easily get

\begin{equation}\label{BD}
h^\infty_n(\underline i) \le \bigl(\ov\BD(\hat H_{\un i }\cap\Gamma)+\e  \bigr) \bigl( \log \sup\eta'\bigr).
\end{equation}
for every  $\underline i = ...,i_0 |$, i.e. $W^u=W^u(p)$
for any $p\in\rho(\underline i)$.

To prove this we cover $W^u$ by pairwise disjoint (except their end points) arcs of the same length equal to  $1/(\sup\eta')^n$ (up to a constant)
and use the definition of box dimension.

A difficulty we shall deal with below, is however to pass in \eqref{BD} to a  uniform over $\un{i}$ version, that
is with $\sup_{\un i} h_n^\infty(\un i)$ in \eqref{uniform}

\smallskip

Notice that
\begin{equation}\label{hh*}
  \ov\BD(\hat H_{\un i}\cap\Gamma) \le t_0=h^*/(-\chi_{\mu_{t_0}}(\lambda'))
  \end{equation}
  $$
\le h^*/(-\sup\log\lambda') < h^* / \log \sup\eta'.
$$
The first inequality  uses the "Lipschitz holonomy" along $\hat{\cW}^u$\;\footnote{Formally
 this is not even
a holonomy, because of intersections of the leaves.
However it is Lipschitz in the sense of varying all the lengths of the uniformly transversal sections
of each strip $\hat H_n$ for all $n$, by at most a common factor.}
(in fact only local) between an arbitrary $\hat W^s$
and $\hat W^u=\hat W^u(p)$.
We shall prove it more precisely below:

\smallskip

Let $p'\in W^s \cap \Lambda$ be such that
$p'=p'(\un i')$ is the only point of the intersection $\rho(\un i')\cap W^s$. Assume that $i'_0\not=i_0$.
Denote by $A(f)$ supremum over all $p,p'$ as above of the number of the intersection points
of $\hat W^u(p)$ and $\hat W^u(p')$.
It is finite by the transversality assumption, see e.g. \cite[Prop. 4.6]{Rams2}.

For every $n\ge 0$ we have, due to $\nu<\lambda$,
$$
\diam (\hat H_n(p')\cap W^s)\le \Const \lambda_n(f^{-n}(p'))
$$
for $n$ large enough to kill a twisting effect which may be caused by $\frac{partial \nu}{\partial y}$.
Hence due to the transversality assumption,
$$
\diam \Comp (\hat H_n(p')\cap W^u(p))\le \Const \lambda_n(f^{-n}(p')).
$$
for every component Comp of the intersection.

\bigskip

By the definition of $t_0$ we have,  summing over all $(i'_{-n},...,i'_0|) $ with $i_0'\not=i_0$,
$$
\sum_{n, p'} \lambda_n (f^{-n}(p'))^t \le C(t)<\infty \ \ {\rm  or} \ \ =\infty
$$
for $t>t_0$ and constant $C(t)$ or $t<t_0$ respectively.

For each $r>0$ and $\hat q\in \hat W^u(p))\cap \Gamma$, where $q\in\rho(...,i'_0|)$ find $n=n(q)$ the least integer such that the
length  satisfies $l(\hat W^u\cap \hat H_{i'_{-n},...,i'_0})\le r$; by the length (denoted above by diam) we mean here the length of the projection by $\pi_x$ to $\R$ (of course we can alternatively consider the lengths in $\hat W^u(p)$ or $W^u(p)$).

Denote $\hat W^u\cap \hat H_{i'_{-n},...,i'_0}$ by $I(q,r)$.
Consider in $\hat W^u$ the ball (arc) $J(q,r)=B(\hat q,r)$. 
Choose a family $J(q_k,r)$ of the arcs of the form $J(q,r)$ covering $\hat W^u\cap\Gamma$, having multiplicity at most 2, namely that each point in $\hat W^u$ belongs to at most 2 arcs. Then $I(q_k,r)\subset J(q_k,r)$ for all $k$.
On the other hand by the definition of $n(q)$ there is a constant $K$ such that $K  l(I(q_k,r))\ge l(J(q_k,r))$.

Finally notice that for two different $q_k$ and $q_{k'}$ it may happen that $n=n(k)=n(k')$ and the $n$-th codings $i_{-n},...,i_0$ are the same;
in other words the $n$-th horizontal cylinders coincide. Then however $J(q_k)$ and $J(q_{k'})$ intersect so the coincidence of these codings may happen only for at most two different $k$ and $k'$.

Thus, for all $t>t_0$,
\begin{equation}\label{t>t0}
\sum_k (2r)^t \le K^{-t} \sum_k l(I(q_k,r))^t \le
\end{equation}
$$
2 \Const K^{-1}\sum_{n,k}
\lambda_n(f^{-n}(q_k))^t \le \Const C(t) <\infty.
$$

Hence, as our estimates hold for every $r>0$, we obtain the first inequality in \eqref{BD}
$$
\ov{\BD}( W^u(p)\cap \Gamma) \le t_0.
$$

\medskip

This has been Moran covering type argument.

\smallskip

\emph{Another variant of this proof} would be to consider for each $n$ the partition of $W^u$ into $d_n:= 2\pi/r_n$ arcs of length $r=r_n=1/([\sup{\eta'}^n]+1)$
and consider the family of those arcs  which intersect $\Gamma$. Denote them by $J_k$. For each $k$ choose an arbitrary $q_k\in \Lambda$ such that $\hat q_k \in J_k\cap \Gamma$ and $q_k$ belongs to some $\rho(\un i')$ with $i_0\not=i'_0$. Then consider $I_k$ as above. Finally notice that each interval $I_k$ as shorter than $r_n$ for $n$ large enough, can appear at most twice.

\medskip

Finally, by \eqref{t>t0}, the estimate \eqref{BD} is uniform, that is $n$ for which it holds is independent of $\un i$.
 ndeed, for each $\un i$ and $n$ we obtain for $r=1/(\sup\eta')^n$, denoting $\e=2(t-t_0)$,

$$
(\exp n h_n^\infty(\un i)) (2r)^{t_0+\e} \le (2r)^{\e/2}\Const C(t)<1
$$
for $n$ large enough.

So $h_n^\infty \le (t_0+\e)(\sup\eta')$ for each $\e>0$ and $n$ large enough, hence
$h^\infty\le t_0 \sup\eta'$. By \eqref{hh*} we have $h^* > t_0 \sup\eta'$. Thus
$h^\infty < h^*$.

\end{proof}

The simplest uniform dissipation, namely if  ${\eta}^{-1}_n \equiv d^{-n}$, provides
the partitions of $S^1$ into arcs of equal lengths to be used in estimating BD. In more general cases
the partitions of $S^1$ into arcs between consecutive $\eta^n$ preimages of a fixed point
cause difficulties and a necessity to assume the half-uniform dissipation assumption using
$\sup \eta'$ rather than $\eta'$. They will be overcome in Section \ref{HD} in Proof of Theorem \ref{Main Theorem}
by replacing $h$ by certain $h^{\rm{reg}}$ by restricting in the definition to   $\mu_{t_0}$-Birkhoff regular points with respect to $\log \eta'$, restricting $\Gamma$.

\bigskip

Now, assuming the uniform dissipation, using $h<h^*$ following from Lemma \ref{h<h*} and Proposition \ref{hh},
we can the following
\begin{lemma}\label{Lip_full_meas}
Assume that the transversality and the uniform dissipation condition (introduced in Proposition \ref{hh}) hold. Then $\mu^s_x(L^s)=1$ for $(\pi_x)_*(\mu)$-a.e. $x\in S^1$,
where $\mu^s_x$ are conditional measures of $\mu$ for the partition of $\Lambda$ into $W^s_x\cap\Lambda$.

The same holds for all $x$, where $\mu^s_x$ is $(h_{x',x})_*(\mu^s_{x'})$ for $x'$ where $\mu^s_{x'}$ has been already defined as a conditional measure.
\end{lemma}

Remark that for $\mu$-almost all $p, p'$ the holonomy $h_{x,x'}$ for $x=\pi_x(p), x'=\pi_x(p')$
maps $\mu^s_x$ to $\mu_{x'}^s$, i.e. $(h_{x,x'})_*(\mu^s_x)=\mu^s_{x'}$.

Note that these measures coincide also
with the factor measure $\mu^-:= \Phi_*(\mu_{t_0})$ where $\Phi$ maps the two sided to the one-sided shift to the right,
on the space of sequences $(...,i_{-n},...,i_0)$ (projected to $f^{-1}$ by $\rho$).
In fact we can write $\mu$ in place of $\mu^-$ considering its restriction to the $\sigma$-algebra generated
by horizontal cylinders.

Then the assertion of Lemma \ref{Lip_full_meas} says that $\mu^-(\Phi(L^s))=\mu(L^s)=1$.

\medskip

\begin{proof}
By Shannon-McMillan-Breiman Theorem \cite{PUbook} applied for $f^{-1}$ and by ergodicity, denoting
as before by $H_n(p)$ the horizontal cylinder of level $n$ containing $p$ 
we get
$$
{1\over n+1}\log \mu( H_n(p))\to h^*
$$
 for $\mu$ almost every $p\in \La$, so
for every $\e>0$ and $n$ large enough

\begin{equation}\label{Shannon}
\exp -(n+1)(h^* + \e) \le \mu(H_n(p)) \le \exp -(n+1)(h^* - \e).
\end{equation}

Given arbitrary $\e>0$ and $n$, denote the  set
of $p$'s  ( a union of $H_n(p)$'s) where (\ref{Shannon}) does not hold, by $Y_{\e, n}$.
Thus, the irregular set
\begin{equation}\label{limsup}
Y^{\rm irreg}_\e:=\limsup_{n\to\infty}Y_{\e,n}= \bigcap_n \bigcup_{k\ge n} Y_{\e,k}
\end{equation}
has measure $\mu$ equal to 0. Its $\e$-regular complement
$\liminf_{n\to\infty}X_{\e,n}= \bigcup_n \bigcap_{k\ge n} X_{\e,k}$ for
$X_{\e.k}=\La\setminus Y_{\e,k}$ has full measure $\mu$ for each $\e$.

\smallskip

Remark that for our Gibbs measure we can use Birkhoff's Ergodic Theorem for $f^{-1}$ and $\log \lambda'$ in place of Shannon-McMillan-Breiman:
$$
\Const^{-1} \exp ((n+1)(t_0 + \e )\chi_\mu(\lambda') ) \le
\Const^{-1} (\lambda^-_{n+1} (p))^{t_0}\le
\mu( H_n(p))\le
$$
$$\Const (\lambda^-_{n+1} (p))^{t_0}\le
\Const \exp ((n+1)(t_0-\e )\chi_\mu(\lambda') ).
$$
where $\chi_{\mu}(\lambda') =\int \log\lambda'\, d\mu$.

\medskip

Denote  by $\cZ_{\e,n}$ the family of all horizontal cylinders $H_{2n}$, whose $\pi_{x,y}$-projections
 $\hat H_{2n}$
have "horizontal extensions" to $(-L2\pi, (L+1)2\pi)$ intersecting $\hat f^n( \Gamma)$
 and else which are in $X_{\e,2n}$. The constant $L$ is the one that appeared in \eqref{SL}.

 \emph{Claim.} The upper bound of $\#\cZ_{\e,n}$ is roughly
$\exp (n h + (n+1)h^*)$  (`roughly' means: up to a factor of order at most $\exp n\e$).

\smallskip

 Indeed, the number $\exp n h$ comes from $f^n(H)$ for each
 $H\in H(n)$, more precisely from $\hat{f}^n$-images of the rectangles
 $\hat H\cap \hat V_{|i_1,...,i_n}$ counted in
 Definition \ref{top_ent}.
The number $\exp (n+1)h^*$ comes from the number of `regular' $H\in H(n)$ roughly
 as in \eqref{Shannon}, that is being $f^{n+1}$-images of `regular' $V_n$. More precisely of those $H$'s for which  some $H_{2n+1}\subset
 f^n(H)$ are in $X_{\e,2n+1}\cap X_{\e,n}$ \footnote{Notice that this does not depend on the choice of $H_{2n+1}$, due to \eqref{chain2}. Notice also that $H$ satisfying this, need not exhaust all $H$ satisfying
 \eqref{Shannon}.}. The measure $\mu$ of each such $H$ is lower bounded for $p\in H$ by

 \begin{equation}\label{chain2}
 \Const \la^-_{n+1}(p)^{t_0}=
 \Const \la^-_{2n+1}(f^n(p))^{t_0}
 / \la^-_n(f^n(p))^{t_0} \ge
 \end{equation}
 $$
 \exp ((n+1) \chi_\mu(\la')-(3n+1)\e)t_0 =  \exp (- (n+1) h^*)  \exp (- (3n+1) \e t_0 ),
 $$
  \medskip

  \noindent due to chain rule
 $f^{-(n+1)}(p) =  f^{-2n+1} (f^n(p)) \circ  f^n(p)$,
   compare \eqref{chain1}, and Gibbs property
 of $\mu$ (used already above to reformulate \eqref{Shannon} to the language of $\la'$).

The number of our `regular' $H_n$'s is bounded by the reciprocal of the bound in \eqref{chain2}.
Thus,
$$
\#\cZ_{\e,n}\le
\exp n (h+\e)  \exp ( (n+1) h^*)  \exp ( (3n+1) \e t_0 ).
$$
The claim has been proven.

\medskip

Thus,
\begin{align*}
 \sum_{H\in \cZ_{\e,n}} \mu(H) & \le
(\sup_{H\in \cZ_{\e,n}} \mu(H)) (\#\cZ_{\e,n}) \\
& \le  \exp (2n+1)(-h^* + \e)  \exp(n (h+\e) + (n+1)(h^*+3t_0\e) \\
& \le    \exp \Bigl( (n+1)(h-h^*) + (3n(1+t_0)+2)\e \Bigl).
\end{align*}

So, for an arbitrary $\e>0$ small enough, 
$$
\lim_{N\to \infty} \mu(\bigcup_{n\ge N}(\bigcup \cZ_{\e,n} \cup Y_{\e,n}  ))=0,
$$
hence $\mu(NL^w)=0$, see Notation \ref{NLw}.

\end{proof}

\section{Hausdorff dimension}\label{HD}

\begin{theorem}[Theorem A]\label{Main Theorem}
Assume $\chi_{\mu}(\nu')<\chi_{\mu} (\lambda')< -\chi_{\mu}(\eta')$.
Then, for $\HD$ denoting Hausdorff dimension, for $\Lambda_x$ denoting $\Lambda\cap (\{x\}\times \D)=\La\cap W^s_x$

1. $\HD(\Lambda_x)=t_0$ for every $x$ and

2. $\HD(\Lambda)=1+t_0$,

\noindent where $t_0={h^* \over -\chi_{\mu_{t_0}}(\lambda')}$.

\end{theorem}

First we prove this Theorem under stronger assumptions:

\begin{theorem}[Theorem A, uniformly dissipative setting ]\label{Affine Main Theorem}

Assertions of Theorem \ref{Main Theorem} hold if $\sup \lambda' < 1/\sup \eta'$ 

\end{theorem}

We start with a general (not only in the uniformly dissipative case)
\begin{lemma}\label{upper}

For every $p\in\Lambda$ all $r>0$ and balls (discs) $B^s$ in the stable manifold $W^s_{\pi_x(p)}$

$$\mu^s_{\pi_x(p)}(B^s(p,r)) \ge \Const (\diam B^s(p,r))^{ t_0}.
$$
In particular

$$
\liminf_{r\to 0} {\log\mu^s_{\pi_{x(p)}}(B^s(p,r)) \over \log \diam (B^s(p,r))}\le t_0.
$$
One can even replace liminf by limsup.

\end{lemma}

\begin{proof} Lemma \ref{upper} follows from

\begin{equation}
\mu^s_{\pi_x(p)}(B^s(p,  \lambda_n(f^{-n}(p)))) \ge
{\Const} (\lambda_n(f^{-n}(p)))^{t_0}.
\end{equation}
One uses the definition of $\mu^s$ (Gibbs property) and the fact that
the diameter of each $B^s(p,r)$ is comparable to $\lambda_n(f^{-n}(p))$, by bounded distortion.
The conditional measures $\mu^s$ were discussed after the statement of Lemma \ref{Lip_full_meas}
\end{proof}

More sofisticated is the opposite inequality:

\begin{lemma}\label{lower}
For $\mu$-a.e. $p\in\Lambda$ the local dimension satisfies
$$
\delta^s:=\liminf_{r\to 0} {\log\mu^s_{\pi_{x(p)}}(B^s(p,r)) \over \log \diam (B^s(p,r)}\ge t_0.
$$
\end{lemma}

\begin{proof}
One uses Ledrappier-Young formula \cite{LY}
$$
h_\mu(f)=\delta^{ss} (-\chi_{\mu}(\nu')) + (\delta^s-\delta^{ss}) (-\chi_\mu(\lambda'))
$$
and the fact that $\delta^{ss}=0$ since
for $p\in L^s$ the local manifold $W^{ss}$ consists only of the point $p$, see Lemma \ref{Wss}.
Remember also, Lemma \ref{Lip_full_meas}, that $\mu_x^s(L^s)=1$ for all $x\in S^1$, hence $\mu(L^s)=1$.
So
$$
h_\mu(f)=\delta^{s} (-\chi_{\mu}(\lambda')),
\; {\rm hence} \; \delta^s= h_\mu(f)/-\chi_\mu(\la')= t_0.
$$
\end{proof}

\begin{proof}[Proof of  Theorem A, uniformly dissipative setting]

\medskip

Step 1.
$\HD(\Lambda_x)=t_0$ for every $x$ follows  from Lemmas \ref{upper} and \ref{lower} and from Frostman Lemma,
see \cite[Theorem 8.6.3]{PUbook}.

 \medskip

 Step 2.
 Since by Lemma \ref{h<h*} and Proposition \ref{hh} $h<h^*:=h_{\mu_{t_0}}(f)$, we know by Lemma \ref{Lip_full_meas} that there exists
 $x$ (in fact for all $x$) $\mu^s_x(L^s)=1$. By Lemma \ref{holonomy_Lip} all the holonomies $h_{x,x'}$ for
 $0\le x'\le 2\pi$ are locally
 bi-Lipschitz on $L^s(x)$.
 Change the coordinates on $\Lambda$ by  $F(x', y, z):= (x',h_{x,x'}^{-1}(y,z))$, mapping
 $\Lambda$ to the cartesian product $[0,2\pi) \times \Lambda\cap W^s_x$. Then this change is locally Lipschitz on
 $\bigcup_{0\le x' \le 2\pi}h_{x,x'}(L^s_x)=L^s$.
 Hence $\HD(\Lambda)\ge \HD(L^s)=1+HD(L^s_x)= 1+t_0$.

\smallskip

More precisely $F$ is locally Lipschitz, in the sense that there exists $L>0$ such that for every $p\in L^s$ there exists measurable $r(p)>0$ such that for every $r\le r(p)$ and $q\in B(p,r)$, \; $\dist (F(p), F(q)) \le L \dist (p,q)$. This is sufficient to non increase dimension by splitting the space into a countable number of pieces.

\medskip

Step 3. The opposite inequality is implied by Lemma \ref{upper}. Indeed, notice that for every $r=\lambda_n(p)$ and $x'\in B(x(p), r)$ we have

  $$
  \{x'\} \times B^s(\pi_{y,z}(p), (C+1)\cdot r)\supset h_{x(p),x'}(B^s(\pi_{y,z},r),
  $$
 where
 $C:= \tan \measuredangle (E^u,E^s)$. Hence

 $$
 \mu^s(\{x'\} \times B^s(\pi_{y,z}(p), (C+1)\cdot r))\ge \mu^s  h_{x(p),x'}(B^s(\pi_{y,z},r).
$$
Hence, for
\begin{equation}\label{skew}
\hat\mu:=d\mu^s_x d\Leb_1(x),
\end{equation}
for every $p\in \Lambda$,
$$
\hat\mu(B(p,(C+1)r))\ge \int_{-r}^r ( \mu^s_{x'} (H_n(p)\cap \Lambda_{x'}) d \Leb_1(x')\ge r\cdot\Const \lambda_n^{t_0}.
$$
and in conclusion $\hat\mu(B(p,Cr))\ge \Const r^{1+t_0}$ yielding the required upper estimate $\HD(\Lambda)\le 1+t_0$.

\end{proof}

\begin{definition}\label{def_regular}
A point $p=\rho(...i_{-n},...,i_0,...,i_n,...)\in\Lambda$ is said to be Birkhoff
$(\xi, \e, N)$-backward regular for an arbitrary $\e>0$ and for $\xi=\nu,\lambda$ or $\eta$, if for all $n\ge N$

\begin{equation}\label{regular}
\exp n(\chi_\mu(\xi)-\e)
\le \xi_n^-(p)) \le
\exp n(\chi_\mu(\xi)+\e)),
\end{equation}
see Notation \ref{chain}.

Analogously $p\in \Lambda$ is said to be Birkhoff $(\xi, \e, N)$-forward regular if the above
estimates hold for
$\xi_n^+(p))$ in place of  $\xi_n^-(p))$.

When we mean just  $\eqref{regular}$ we say \emph{$(\xi, \e, n)$-forward (backward) regular},
omitting "Birkhoff".  Compare Shannon-McMillan-Breiman property in Proof of Lemma \ref{Lip_full_meas}

By bounded distortion the property \eqref{regular} for $p=\rho(...i_{-n},...,i_0,...,i_n,...)$
depends only on $(i_{-n},...,i_0)$, provided we insert constant factors before exp, so it can be considered as a property of a horizontal cylinder $H_n$.  Analogously for the forward regularity this is a property of vertical cylinders $V_n$. We call these cylinders $(\xi, \e, n)$-forward or backward regular
and all other points or level $n$ cylinders \emph{irregular}.

\end{definition}

\medskip

\begin{proof}[Proof of Theorem A] We shall modify (generalize) the definition of irregular sets $Y_{\e,n}$ in  Lemma \ref{Lip_full_meas} and follow the strategy of the proof of that lemma.

\medskip

Recall the notation $\xi_n^-(p):= \xi_n(f^{-n}(p))$.
Notice that for all integers $m>0$

\begin{equation}\label{chain1}
\xi_m^-(f^{-n}(p) = \xi_{n+m}^-(p) / \xi_{n}^-(p)
\end{equation}
hence for $p$ being $(\xi,\e, k)$-backward regular for $k=n$ and $k=n+m$,
we have
\begin{align*}\label{back_reg}
& \exp (n+m)(\chi_\mu(\xi)-\e)) / \exp n(\chi_\mu(\xi)+\e)) \le
\xi_m^-(f^{-n}(p)\le  \\ &
\exp (n+m)(\chi_\mu(\xi)+\e)) / \exp n (\chi_\mu(\xi)-\e)).
\end{align*}
Hence

\begin{equation}
\exp m (\frac{n+m}m  (\chi_\mu(\xi)-\e) -  \frac{n}m (\chi_\mu(\xi)+\e))
\le \xi_m^-(f^{-n})(p)\le ....
\end{equation}

hence
\begin{equation}\label{back_reg_2}
\exp m (\chi_\mu(\xi)-\e(2\frac{n}m + 1))
\le \xi_m^-(f^{-n}(p)) \le
\exp m (\chi_\mu(\xi)+\e(2\frac{n}m + 1)).
\end{equation}

 \bigskip

 For each $n,m\in\N$ denote by $X_{\e,n,m}$ the union of all $H_{n+m}$ horizontal cylinders of  $(\xi,\e,n)$-backward regular points in $\Lambda$ for all $\xi=\nu,\lambda$ and $\eta$ and yet
 $(\la,\e,n+m)$-backward regular.

 Write also $Y_{\e,n,m}:=\Lambda \setminus X_{\e,n,m}$
for irregular sets. 
Now, as in Section \ref{Lip vs geo}, Proof of Lemma \ref{Lip_full_meas}, the idea  is to
remove\footnote{Due to the uniform dissipation assumption we needed to remove there less than here.}
for each $n$
the irregular set
$Y_{\e,n,m}$ for $m$ to be defined later on, and estimate the number of remaining cylinders $H_{n+m}$  which are \emph{regular contaminated} by other regular cylinders in the sense below \eqref{contaminated}.

\bigskip

A point (and cylinder) $p\in H_{n+m}$ regular as above is said to be $(\Gamma^{\reg}_{n,m})$-contaminated if
for ${\tilde p}:=f^{-n}(p)$
\begin{equation}\label{contaminated}
\pi_{x,y}({\tilde p})\in B^u(\Gamma^{\reg},
L_1 {\eta}^{-1}_n(\tilde p)),
\end{equation}
compare Definition \ref{def_strong_Lip}. $B^u$ denotes a  ball in ${\widehat W}^u(\tilde p)$.
The set $\Gamma^{\reg}$ is defined as $\Gamma $ in Definition \ref{Gamma}, but restricted to
$\hat p$ being $\pi_{x,y}$ image of $q=\rho(...,i_0|)$ and $q'=\rho(...,i'_0|)$ such that $f^n(q)$ and $f^n(q')$
are in $X_{\e,n,m}$.

As in Definition \ref{def_strong_Lip} we can say equivalently that $\hat V_n(\tilde p)$
is $\Gamma^{\reg}_{n,m}$-contamina\-ted if it does not satisfy \eqref{SL}, with $\Gamma$ replaced by $\Gamma^{\reg}_{n,m}$. We can say also that the rectangle $\hat H_m \cap \hat V_n$ is contaminated,
as in Definition \ref{top_ent}. See also Subsection \ref{strategies}. 

\bigskip

Here
it is comfortable to look for $m>0$ as small as possible so that
\begin{equation}\label{m(n)}
\lambda^-_m(f^{-n}(p) < (\eta^-_n(p))^{-1},
\end{equation}
compare \eqref{nm} later on.
Taking in account that both $f^n(q)$ and $q$ are in
$X_{\e,n,m}$
we obtain using
\eqref{back_reg_2}
the sufficient condition
$$
\exp m (\chi_\mu(\lambda)+\e(2\frac{n}{m} + 1)) <
\exp n (-\chi_\mu(\eta)-\e)
$$
It follows that for $\epsilon>0$ small
it is sufficient

\begin{equation}\label{mN}
m/n \approx \chi_\mu(\eta')/(-\chi_\mu(\lambda')) + \e'
\end{equation}
with $\e'>0$ also small.

\

Summarizing: for given $H_m(\tilde p)$ with $p=f^n(\tilde p)\in X_{\e,n,m}$ 
we define

\noindent $h^{\reg}_n:=\frac1{n+1}\log Z_n$ where $Z_n$ is the number of $\Gamma^{\reg}_{n,m}$-contaminated $\hat V_n$
in $\hat H_m(\tilde p)$
(by $H_m(\tilde q)$ with the $i_0$ symbols different from the one for  $H_m(\tilde p)$,
 with
 $q=f^n(\tilde q)\in X_{\e,n,m}$).
\smallskip

\noindent The number $Z_n$ is bounded by a constant times the number of $H_m$ above,
taking in account $L$ in \eqref{SL} and the observation that regular $H_m$,
as "thinner" than $V_m$ can intersect at most two (neighbour) $V_m$'s.

So
$$
\exp n h^{\reg}_n \le \Const \exp m h^*,
$$
hence using \eqref{mN},
\begin{equation}\label{eq:h*h}
h^{\reg}\le h^* (\chi_\mu(\eta')/(-\chi_\mu(\lambda'))) + \e'.
\end{equation}

The rest of the proof repeats Proof of Theorem \ref{Affine Main Theorem}

In particular by Birkhoff ergodic theorem for an arbitrary $\e>0$

\noindent $\mu(\lim\sup_{n\to\infty} Y_{\e,n})=0$ and the complementary set  in $NL^w$,
for $\e>0$ small enough,
where $h^{\reg}<h^*$,  has measure $\mu$ also equal to 0.

\end{proof}

The above proof finishes also the proof of Theorem D in the general setting, saying that
$\mu_x^s(NL^w)=0$, compare  Lemma \ref{Lip_full_meas} in the uniform dissipation case.
Compare also \eqref{limsup}.

\bigskip

\section{Packing measure}\label{Packing}

\bigskip

For the definition of \emph{packing measure} we refer the reader to \cite[Section 8.3]{PUbook}.
Denote packing measure in dimension $t$ by $\Pi_t$. 

We shall prove the following

\begin{theorem}\label{packing} Under the assumptions of Theorem \ref{Main Theorem} (Theorem A),
 namely if
$$
\chi_{\mu}(\nu')<\chi_{\mu} (\lambda')< -\chi_{\mu}(\eta'),
$$
for the Gibbs measure $\mu=\mu_{t_0}$ on $\Lambda$,
then for every $p\in \Lambda$ it holds
\begin{equation}\label{Pack}
0<\Pi_{t_0}(W^s(p)\cap \Lambda)<\infty.
\end{equation}
Moreover
the density $d\Pi_{t_0}/d\mu^s_{t_0}$ is positive $\mu^s_{\pi_x(p)}$-a.e.
(recall that $\mu^s_{\pi_x(p)}$ is the conditional measure on $W^s(p)$).

Also
\begin{equation}\label{PackLambda}
 0<\Pi_{1+t_0}(\Lambda)<\infty
\end{equation}
and moreover
$d\Pi_{t_0+1}/d\mu_{t_0}$ is positive $\mu_{t_0}$-a.e. on $\Lambda$.

\end{theorem}

This generalizes the analogous theorem proved for linear solenoids in \cite{RS}.

\begin{proof}

\bigskip

Step 1. \emph{Regular contaminated rectangles.}

\medskip

Given an arbitrary $\e>0$
denote by $\cH(\e,t)$ the
union of all $H_t$ containing points in $\Lambda$
satisfying the backward regularity condition \eqref{regular} for $t$ (denoted there $n$) and
$\xi=\lambda,\nu$.

Analogously denote by $\cV(\e,t)$ the union of all $V_t$ containing points in $\Lambda$
satisfying the forward regularity condition analogous to \eqref{regular} for $t$ (denoted there by $n$) and $\xi=\eta$ and by $\cV(\e,t)$.

We sometimes call $H_t$ and $V_t$ as above, just \emph{regular}.

\bigskip

Consider an arbitrary  $m\in\N$ and given $\e>0$ define $n=n(\e,m)$ as
the biggest integer $n$ such that, compare \eqref{mN} and \eqref{m(n)},
\begin{equation}\label{nm}
\frac{n}{m} \le \frac{-(\chi_\mu (\lambda')-\e)}{\chi_\mu(\eta')+\e}.
\end{equation}
We consider $\e$ small enough that the latter fraction is bigger than 1.

Later on we shall consider an arbitrary $\a: 0<\a\le 1$ and the integer $[\a n]$ in place of $n$ (the square bracket  means the integer part), sometimes writing just $\a n$. Finally we shall
specify  $\a$. Of course \eqref{nm} is satisfied for $[\a n]$ in place of $n$.

Notice that for an arbitrary $(\la,\e, m)$-backward regular $p\in H_m\subset \cH(\e,m)$, therefore for regular $H_m(p)$; the diameter of its intersection with any
$W^s$ is at most $\exp ( m(1-\e)\chi_\mu(\lambda'))$ (up to a constant related to distortion).

\smallskip

For all $\un i = (...,i_0|)$, writing $\rho(\un i)=W^u(\un i)=W^u$
we obtain the uniform (over $\un i$) estimate \eqref{BD} on $h_n^\infty(\un i)$ as
in Lemma \ref{h<h*} for $W^u$ restricted to the intersection with $\cV(\e,n)$.
We write $h_n^{\infty, \reg}(\un i)$.

Indeed, we can use then for every forward regular $V_n$ the property

\noindent $\diam (W^u\cap V_n) \ge \exp -n(\chi_\mu(\eta')+\e)$.
(We accept that one $\e$ can differ from another if it does not lead to a confusion.)
In \eqref{hh*} we use then $\chi_\mu(\lambda')< -(\chi_\mu(\eta')+\e)$.

Defining $h^{\infty,\reg}:=\lim_{n\to\infty} \limsup_{\un i} h_n^{\infty, \reg}$ analogously to
Definition \ref{top_ent_infty}, we get  for $\e$ small enough
$$
h^{\infty,\reg}<h^*.
$$

\medskip

We obtain the same estimates, in particular $h^{\reg}<h^*$, if in place of $W^u$ , thicken it to $H_m$
restricting ourselves to backward regular $p\in H_m\subset \cH(\e, m)$,
because then, if $p\in V_{[\a n]}$  backward regular, for $n=n(\e,m)$
$$
\diam \pi_{x,y} (H_m(p) \cap W^s) \ll \diam \pi_{x,y}(V_{[\a n]}(p) \cap W^u)
$$
due to
$$
 \exp  m(\chi_\mu(\lambda')+\e) < \exp -[\a n](\chi_\mu(\eta')+\e),
$$
see Definition \ref{top_ent} and the transversality.

\medskip

 In words, the number of forward regular
 vertical cylinders $V_{[\a n]}$  whose $\pi_{x,y}$ projections $\hat V_{[\a n]}$ intersect the "rhombs" $\hat H_m(p)\cap \hat H'_m$ with $H'_m=\rho(i'_{-m},...,i'_0)$, $i'_0\not=i_0$ as in the Definition \ref{top_ent}
 widened by their $L_1$'th neighbours in $\hat H_m(p)$, is bounded by
 $\exp {[\a n]}  (h^{\reg}+2\e)$. This estimate is uniform over our regular $H_m$'s.

 \smallskip

Then their union denoted by $\cV_{[\a n]}  (H_m)$  has measure $\mu$ upper bounded by
 \begin{equation}\label{contaminated_bound}
  \exp {[\a n]} (h^{\reg} +2\e) \exp(- {[\a n]} (h^*-\e))\le
 \end{equation}
 $$
 \exp ( {[\a n]} (h^{\reg}-h^*)+3 {[\a n]} \e)\le
\Const  \exp {[\a n]} (h^{\reg}-h^*+3\e)
$$
again uniformly for regular $H_m$'s,
and for $m$ is large enough, exponentially decreasing as $m\to\infty$, for $\e$ small enough.

\medskip

By Gibbs property the same estimate holds for conditional measure $\mu$ in $H_m$, namely
$\mu(\cV_{[\a n]}  (H_m)\cap H_m / \mu(H_m)$, or just for $\mu$ restricted to $H_m$,  summed over regular $H_m$'s.

\medskip

Step 2. \emph{Close cylinders.}

\medskip

We keep $n, m$ and arbitrary $\a\le 1$  as above and consider an arbitrary integer
$0<k\le m$.
We take care of intersections of $\hat H_m$ with $\hat H'_m$'s with $i_{-k}\not=i'_{-k}$ but $i_t=i'_t$ for
all $t=0,...,-(k-1)$. Consider $f^{-k}(H_m)$ as one of the summands of
the union
$$
H_{m-k}:=H_{i_{-m},...,i_{-k} |}=\bigcup_{i_{-k+1},...,i_0} H_{i_{-m},...,i_{-k} |}\cap V_{| i_{-k+1},...,i_0}.
$$
By the estimate in Step 1 for $m-k$ we cover the union of intersections $\hat H_{m-k} \cap \hat H'_{m-k}$ with margins by a family of  $\pi_{x,y}$-projections of $V_{\a n(\e,m-k)}$'s being $(\eta,\e, [\a n(\e,m-k))]$-forward regular, leaving aside the part covered by irregular vertical ones, for $H_{m-k}$ backward regular,  The union of this family has $\mu$-conditional measure in $H_{m-k}$ bounded by
$$
\Const  \exp [\a n(\e,m-k)] (h^{\reg}-h^*+3\e).
$$

\medskip

So, the union of these $H_{m-k}\cap V_{[\a n(\e, m-k)]}$'s has exponentially shrinking measure $\mu=\mu_{t_0}$
for each $m$ and  for $m-k$ growing from $m_0$ to $m$.  So
 the sums over $k=0,...,m-m_0$  are bounded by a constant independent of $m$, say by $\frac12$.
 By $f$-invariance of $\mu$ the same bound by $\frac12$ holds for $\bigcup_{k=0,...,m-m_0} f^{k}(\cR(m,k))$, where
 $\cR(m,k)$ is the union of all regular $H_{m-k}\cap V_{[\a n(\e, m-k)]}$ above.

 By construction all $f^k(H_{m-k}\cap V_{[\a n(\e, m-k)]})$ (regular and not regular) are unions of `rectangles' $H_m\cap V_{[\a n]}$ because
 $[\a n(\e, m-k)]\le [\a n]$. So their $f^{[\a n]}$-images are unions of entire $H_{[\a n]+m}$'s.

The conclusion is that for each $m, n=n(\e,m)$ the union $\bH^{\rm b} (m, \reg)$ of all `regular' (more precisely
$f^{[\a n]+k}$-images of regular, see also Step 3) $H_{m+[\a n]}$'s whose $\pi_{x,y}$-projections intersect at most \underline{bounded} number of  others
in the same $H_{[\a n]}$,
not only `regular'\footnote{ In Proof of Lemma \ref{Lip_full_meas} and in Proof of
Theorem  \ref{Main Theorem} (Theorem A) we just removed irregular horizontal cylinders, with union given $n$ of measure tending to 0 by  Birkhoff ergodic theorem,
and eventually with $\limsup_{n\to\infty}$ of measure 0. These unions could be even proven  to be of exponentially decreasing measure $\mu$ if we referred to the large deviations Lemma \ref{LD}.
Here we  have additional summing over $k$ which makes these irregular unions of measure larger than a positive constant for all $n$ and depending on $n$. If we removed them, we would risk  removing everything.}
$\hat H'_{m+[\a n]}$'s,
together with the union of all `irregular' ones, to be estimated in Step 3,
has measure $\mu$ at least $\frac12$.

We write `bounded number' rather than not intersecting at all, since we have not taken care of intersections
of $\hat H_{m-k}\cap \hat H'_{m-k}$ for $m-k< m_0$, i.e. after acting by $\hat f^{k+[\a n]}$ the intersections of $\hat H'_{m+[\a n]}$ and $\hat H_{m+[\a n]}$ being `close neighbour' cylinders with coding different at most on positions
$-(m+[\a n]),..., -(m+[\a n]-m_0)$.


\


Step 3. \emph{Remote cylinders.} \ We discussed above $\hat H_{m+[\a n]}$ intersecting bounded number of $\hat H'_{m+[\a n]}$'s in the same $\hat H'_{[\a n]}$. We do not know how to avoid intersections of $\hat H_{m+[\a n]}$ and $\hat H'_{m+[\a n]}$ having
$(i_{-[\a n]},...,i_0|)$ different from $(i'_{-[\a n]},...,i_0'|)$. However for each  $(\xi,\e, [\a n])$-backward regular for $\xi=\la',\nu'$ and
$p\in H:=H_{m+[\a n]}$ and $p'\in H':=H'_{m+[\a n]}$ in different horizontal cylinders $H_{[\a n]}$ and in each $W^s_x$ we have
\begin{equation}\label{disjoint}
B^s(H\cap W^s_x, C\la_{m+[\a n]}^-(p))\cap B^s(H'\cap W^s_x, C\la_{m+[\a n]}^-(p') =\emptyset
\end{equation}

for arbitrary constant $C>0$, \; $m,n$ large enough, $\e$ small, provided
$$
\exp [\a n]\chi_\mu(\nu') \gg \exp (m+ [\a n]) \chi_\mu (\la'),
$$
that is we assume
\begin{equation}\label{alpha}
\a< \frac{\chi_\mu(\eta')}{-\chi_\mu(\nu')+\chi_\mu(\la')}.
\end{equation}

\

Step 4. \emph{Irregular sets.}

\medskip

Above `regular' means: in $\bigcup_{0\le k\le m-m_0} f^{[\a n(\e,m-k)+k}(\cR(m,k))$, that is in
$\bH_1(m,k,\reg):=f^{[\a n(\e,m-k)+k]}   (\cH (\e,m-k)) $  and
in

\noindent $\bH_2(m,k,\reg):=f^{[\a n(\e,m-k)+k]}   (\cV (\e, n(\e, m-k)))$
for all $k=0,...,m-m_0$, and additionally not in $\bH_3(t, \irreg)$ for all $t$ large enough, see below.

\medskip

Denote the complementary, `irregular', sets in $\bigcup\{H_{m+[\a n(\e,m)}\}$
by

\noindent $\bH_1(m,k,\irreg)$ and $\bH_2(m,k,\irreg)$.




\smallskip

Due to large deviations Lemma \ref{LD-simple}, see \eqref{LD},
$$
\mu ({\bH}_1 (m,k, \irreg)) \le \Const \exp - (m-k)\tau
$$
and
$$
\mu ({\bH}_2 (m,k, \irreg)) \le \Const \exp - [\a n(\e,m-k)]\tau
$$
for a constant $\tau>0$ depending on $\e$, and the functions $\la'$ and $\eta'$. 

When we take unions over $0\le k\le m-m_0$ we obtain an upper bound for measure $\mu$
of the unions ${\bH}_i (m,\irreg)$ for $i=1,2$ of these `irregular' $\bigcup\{H_{m+[\a n(\e,m)}\}$'s by
a small constant, say $\frac18$, for $m_0$ large enough, for each  $n$ (formally for each $m$, but then each $n$ is counted by a bounded number of times).

\bigskip

Finally we distinguish another irregular set in $\bigcup\{H_{m+[\a n(\e,m)}\}$,
 for each $m$ large enough
namely the complement $\bH_3 (m, \irreg)$ of the set of all $(\xi, \e, m+[\a n)]$-backward regular cylinders $H_{m+[\a n]}$
for  $\xi=\nu$ and $\lambda$. By Birkhoff ergodic theorem we can assume that $\mu(\bigcup_{t>N}\bH_3 (t, \irreg)) <\frac18$. for $N$ large enough (compare $\mu(\limsup_{n\to\infty}Y_{\e,n})=0$ in Proof of
Lemma \ref{Lip_full_meas}.)

\

Step 5. \emph{The conclusion in stable manifolds.}

\

Denote $\bH(m,\reg):=\bH^{\rm b} \setminus \bigcup_{i=1}^3 \bH_i(m,\irreg)$.
To conclude the proof of
our theorem use now \cite[Lemma 3]{RS} for the conditional measure $\mu_x^s$ on $W_x^s$.
It yields in our case that due to $\mu(\bH (m,\reg))\ge \frac12 - 3\cdot\frac18=\frac18$ for  $m$ large enough,
hence $\mu^s_x(\bH (m,\reg)\cap W^s_x)\ge \Const \cdot \frac18$, for
a positive measure $\mu_x^s$ subset $\bW$ of $W_x^s$,  for every $q\in \bW$ there is
 a sequence $m_j$ such that  $H_{m_j+[\a n_j]}(q)\in \bH(m_j,\reg)$. In particular there is a sequence of `regular' horizontal cylinders containing $q$ of level tending to $\infty$,
 whose $\pi_{x,y}$-projections
are each at most boundedly intersecting the family of projections of other horizontal
cylinders of the same level, provided they are both in $H_{[\a n]}(q)$.

Therefore, for $q\in \bW$ by Gibbs property, due to $\chi_\mu(\nu')<\chi_\mu(\lambda')$,  \eqref{alpha}, and regularity,
there is a sequence $r_j\searrow 0$ such that
\begin{equation}\label{Frostman}
\mu^s_x (B(q,r_j))\le C r_j^{t_0}.
\end{equation}
Hence $\Pi_{t_0}(\bW)\ge \Const C^{-1} \mu_x^s(\bW)$, see e.g. \cite[Theorem 8.6.2]{PUbook}.

The density $d\Pi_{t_0}/d\mu_x^s $ is positive $\mu$-a.e.
since the set $\bW$ can be found of measure $\mu$ arbitrarily close to 1.
This can be achieved by replacing the constants $\frac12$ and $\frac18$
by arbitrarily small positive constants, by increasing $m_0$ adequately.
This increases the allowed bound of the multiplicity of intersections of $H_{m+[\a n]}$'s,
thus increasing $C$.

Another variant of this part of the proof is to use ergodicity of $f$.

\smallskip

Finally the existence of an upper bound of $d\Pi_{t_0}/d\mu_x^s $,
 in particular finiteness of $\Pi_{t_0}(W^s\cap \Lambda)$ follows from the
uniform boundedness from below of $\frac{\mu^s_x B(q,r)}{r^{t_0}}$ for $r$ small enough, see Lemma \ref{upper}.
We again refer to \cite[Theorem 8.6.2]{PUbook}.

\bigskip

Step 6. \emph{Packing measure in $\Lambda$.}

\medskip

To prove $0<\Pi_{1+t_0}(\Lambda)$ in \eqref{PackLambda}
notice that 
for an integer $n_0$
and every $q\in \bW$,  every $k_j:=m_j+[\a n_j]$ as at the beginning of Step 4
and every $p_1,p_2\in W^u(q)$ we have the following inclusion of intervals
$$
\pi_{x,y} (H_{k_j+n_0}(q)\cap W^s_{\pi_x(p_1)}) \subset \pi_{x,y}
(H_{k_j}(q)\cap W^s_{\pi_x(p_2)}),
$$
provided $\dist(p_1,p_2)<r'_j$, where $r'_j:=\lambda_{k_j+n_0}^- (q)$.

In words:
each square of sides of order $r'_j$, namely
$$
[\pi_x(p_1),\pi_x(p_2)] \times (\hat H_{k_j+n_0}(p_1)\cap W^s_{\pi_x(p_1)})
$$
is a subset of a
piece of $\hat H_{k_j}(q)$ of length $r'_j$ (along $x$-axis),
with vertical (along $y$) sections of length of order $r_j$,
where  $r_j:=\lambda^-_{k_j} (q)$.

Hence, for "skew product" $\hat\mu$ as in \eqref{skew}
$$
\hat\mu (B(p_1, \Const r'_j)\le \Const r_j^{1+t_0}.
$$
We used here the fact that for $\Const >0$ small enough
$$
\pi_{x,y} (B(p_1, \Const r'_j)\cap W^s_{\pi_x(p_1)}) \subset \hat H_{k_j+n_0}(p_1)  \cap    W^s_{\pi_x(p_1)}.
$$

Applying Frostman lemma finishes the proof of left-hand side inequality of \eqref{PackLambda}.
The right-hand side inequality follows from Lemma \ref{upper}.

\end{proof}

\medskip

\begin{remark}

When we take $f^k$ or $f^{n+k}$ image, the conditional measures stay the same by the $f$-invariance of $\mu$.

The phenomenon which manifests and helps is the affinity of the mapping when we measure distances with respect to invariant measures after passing to conditional measures on unstable foliation.

\end{remark}

\begin{remark}\label{rem:5.3}

Notice that in estimating from below the local dimension $\delta^s$ of $\Lambda \cap W^s(p)$ for a.e. $p$
we referred to Ledrappier-Young formula, using $W^{ss}_{loc}(p)\cap\Lambda=\{p\}$, Lemma \ref{Wss}.

In fact we knew there only that $\hat H_{2n}(p)$ did not intersect
$\hat H_{2n}(p')$ such that $H_{2n}(p')\subset H_{n-1}(p)\setminus H_n(p)$, but we did not exclude
the intersecting for $H_{2n}(p')\subset H_n(p)$. To avoid intersections we split $H_n(p)$ into $H_{n+1}(p)$ and the complement, splitting both into $H_{2(n+1)}$ and getting disjointness for
$H_{2(n+1)}(p')\subset H_{n}(p)\setminus H_{n+1}(p)$. Etc., splitting $H_{n+1}(p)$, $H_{n+2}(p)$ ... .
This allowed  the local disjointness of
$\hat W^u$'s as in
the preceding paragraph.

We coped with the disjointness of entire $\hat H_n$'s in Section 5 on packing measure, but for each $W^u$ the disjointness of the consecutive cylinders containing it has been proved
only for a sequence of $n$'s.

For a sequence of $n$'s  multiple self-intersections happen,
thus leading to a proof that Hausdorff measure of each $W^s\cap \Lambda$ in dimension $t_0$ is 0,
see \cite{RS} in the affine case. See the next section.

\end{remark}

\bigskip

\section{Hausdorff measure}\label{hausdorff}

\begin{theorem}\label{Hmeas} For $f$ like in Theorem A \ref{Main Theorem}
Hausdorff measure in dimension $t_0$, denoted by ${\rm HM}_{t_0}$ on each $W^s_x$ satisfies
\begin{equation}\label{HM-singularity1}
{\rm HM}_{t_0}(W^s_x\cap \Lambda)=0.
\end{equation}
and
\begin{equation}\label{HM-singularity2}
{\rm HM}_{1+t_0}(\Lambda)=0
\end{equation}
\end{theorem}

\begin{proof}
Two
horizontal cylinders $H_{n,1}, H_{n,2}$ of level $n$
are said to   \emph{overlap}  if for each $x\in S^1$ the set $\hat H_{n,1}\cap \hat H_{n,2} \cap \hat W^s_x$ is non-empty (remember that the `hat' means the projection by $\pi_{x,y}$).

Such a pair exists. Indeed take $W^u(p_1)$ and $W^u(p_2)$ for $p_1,p_2\in\Lambda$ so that their $\pi_{x,y}$ projections intersect at $\pi_{x,y}(p_1)=\pi_{x,y}(p_2)$\footnote{Such an intersection point exists, see \cite{Borsuk}}. Thicken them by
$H_{k,1}, H_{k,2}$ and consider vertical $V_m$ containing $p_1$ and $p_2$. If $m$ is large enough then
$H_{k+m,1}:=f^m(H_{k,1})$ and $H_{k+m,2}:=f^m(H_{k,2})$ overlap.

$H_{n,0}$ is said to have an order $d$ overlap if there exist  $H_{n,i}, i=1,...,d$ horizontal cylinders of order $n$ such that for all $x\in S^1$ and $i=1,...,D$
$$
\hat H_{n,0}\cap \hat H_{n,i} \cap \hat W^s_x  \not= \emptyset .
$$
Such a family exists for every $d$ and some $n$. Indeed.  Suppose we found already $H_{n,0}$ having an order $d-1$ overlap
with  $H_{n,i}, i=1,...,d-1$. Take $H_{n,d}$
with $i'_0\not=i_0$, the zero symbols for $H_{n,0}$ and $H_{n,d}$,
so that the intersection $\hat H_{n,0}\cap \hat H_{n,d} \cap \hat W^s_x $ is non-empty, say contains a point $q=(x,y)$. Then consider vertical $V_m$ whose $\pi_{x,y}$ projection contains $q$. Then as above, for $m$ large enough $f^m(H_{n,i}), i=0,1,...,d$ is the required family. (Notice that the latter intersection contains a point in
$\hat\Lambda$ provided $q\in\hat\Lambda$; however we shall not use this observation.)

\smallskip

Choose now an arbitrary Birkhoff forward regular $\tilde p\in H_{n,0}\cap\Lambda$. Replace the overlapping cylinders $H_{n,i}, i=0,...,d$ by $H_{n+k,i}=f^k(H_{n,i}\cap V_k)$ for $V_k\ni p$ and $k$ large, to use  time convergences in Birkhoff ergodic theorem.
So, for $p:=f^k(\tilde p)$ we have 
for each $i=0,...,d$,  $H_{n+k,i}\cap W^s(p)
\subset B^s(p,\Const\lambda^-_{n+k}( p))$. This is so due to $\chi_\mu(\nu') < \chi_\mu(\lambda')$ since then
$\nu^-_k(p) \ll \la^-_{n+k}(p)$ for all $k$ large enough (depending on $\tilde p$). Therefore this property is forward invariant under $f$.

We conclude that for $r=\Const \lambda^-_{n+k}(p))$ for $x=\pi_x( p)$,
and adequate constant $C$,
\begin{equation}\label{Hd}
\mu_x^s(B(p,r))\ge C (d+1) r^{t_0}.
\end{equation}

The set $A(d)$ of these $p$ has positive measure $\mu$ and is invariant under holonomies $h_{x,x'}$.
Therefore, invoking also ergodicity of $\mu$, \eqref{Hd} holds in every $W^s_x$ for $\mu^s_x$-a.e. $p\in W^s_x$ and $r=r(p,d)$. If we consider $A= \bigcap_{d\in\N} \bigl( \bigcup_{n\in \N}f^n( A(d))\bigr)$ then using Frostman lemma we prove \eqref{HM-singularity1} for $\Lambda$ replaced by $A$.
Finally use  Lemma \ref{upper}, by which for each $x$,
for $A':=\Lambda \setminus A$,
$\mu^s_x (A') =0$ implies ${\rm HM}_{t_0}(A')=0$.

Similarly, compare Proof of Theorem \ref{Affine Main Theorem}, Step 3, one proves \eqref{HM-singularity2}.

\end{proof}

\

\section{Final remarks}\label{fin-hausdorff}

\subsection{Large deviations}

\bigskip

We refer to the following lemma, see e.g. \cite[Theorem 1.1]{DK} 

\begin{lemma}(On large deviations)\label{large_deviations}
Let $F:X\to X$ be an open distance expanding map of a compact metric space, see \cite[Section 4.1]{PUbook}.
Then, for any H\"older continuous potential
$\Phi:X \to\R$ let $\mu_\Phi$ denote the unique Gibbs invariant measure for $\Phi$, see \cite{Bowen}.  Consider
 arbitrary H\"older functions $\phi, \psi: X \to \R$.  Then, 
for every $t \in\R$,
$$
\begin{aligned}
\lim_{n\to\infty}\frac1n \log\mu_\varphi\bigg(\Big\{x\in X: & \; {\rm sgn}(t)S_n \psi(x) \ge {\rm sgn}(t) n \int_X\psi\, d\mu_{\varphi+t\psi}\Big\}\bigg) \\
&= -t\int_X\psi\, d\mu_{\varphi+t\psi}+P_{top}(\varphi+t\psi)-P_{top}(\varphi),
\end{aligned}
$$
where by $P_{top}$ we denote topological pressure, see e.g. \cite{PUbook}.

\end{lemma}

Writing $\int\psi\, d\mu_{\varphi+t\psi}- \int\psi\, d\mu_{\varphi}:=\e$
we can rewrite the above formula as follows
$$
\begin{aligned}
\lim_{n\to\infty}\frac1n \log &\,\mu_\varphi\bigg(\Big\{x\in X:  \, {\rm sgn}(t)S_n \psi(x) \ge
{\rm sgn}(t) n (\int\psi\, d\mu_{\varphi+t\psi} + \e)\Big\}\bigg) \\
&= -t\int_X\psi\, d\mu_{\varphi+t\psi}+P_{top}(\varphi+t\psi)-P_{top}(\varphi) := I(\pm\psi, \e),
\end{aligned}.
$$
The latter $I(\e)$ measures the nonlinearity of $t\mapsto P_{top}(\varphi +t\psi)$.

\smallskip

A basic example of such $F$ is $\varsigma:\Sigma_d^+\to\Sigma_d^+$ being the left shift map
on the one-sided shift space with the standard metric
$\dist(\un i,\un i')=\sum_{n\in\N}|i_n-i'_n|d^{-n}$.

Symmetrically one considers the right shift map $\varsigma^{-1}:\Sigma^-_d\to\Sigma^-_d$
on the space of sequences $(...,i_n,...,0|)$. We can consider two-sided sequences or e.g. our solenoid $\Lambda$  identifying sequences with the same future, or past as for our $W^s$'s and $f^{-1}$.
Compare Definition \ref{horizontal and vertical}

\bigskip

In particular the following holds

\begin{lemma}\label{LD-simple}

For every H\"older $\phi$ and $\psi$, for every $\e>0$  there exist $C>0$ and  $\tau>0$ such that for every $n\in \N$
\begin{equation}\label{LD}
\mu_\varphi\bigg(\Big\{x\in X:
|S_n \psi(x) - n\int_{\Sigma_d^+}\psi\,d\mu_\phi|\ge n\e\Big\}\bigg)\le C \exp (-n\tau).
\end{equation}


\end{lemma}

\medskip

In Sections  \ref{Lip vs geo} and \ref{HD}, proving e.g. that $\mu_x^s(NL^w\cap W^s_x)=0$ in Lemma \ref{Lip_full_meas} we did not use \emph{large deviations}.  In Section \ref{Packing} we already did (the qualitative version of Lemma \ref{LD-simple}. Now we shall show  how the usage
of large deviations, Lemma \ref{large_deviations}, allows to
estimate from above Hausdorff dimension of the set in each $W^s_x$ where the holonomy is not locally Lipschitz, thus strengthening Lemma \ref{Lip_full_meas}. See also notation in and after Lemma \ref{holonomy_Lip}

\begin{proposition} For every $W^s_x$ the strict inequality $\HD(  NL^w\cap W^s_x) < t_0=\HD(\Lambda\cap W^s_x)$ holds.
More precisely
\begin{equation}\label{eq:HDNL}
\HD (NL^w\cap W^s_x) \le \inf_{\e>0}\max \{A_\e, B_\e\},
\end{equation}
where

$$
A_{\e}= t_0-(I_\e(\log\la',\e)/\chi_\mu(\la'))/ (1+\frac{-\chi_\mu(\la') -  \e}{\chi_\mu(\eta')-\e})
$$
 and the same with $\log\la'$ replaced by $-\log\eta'$.  Also with the latter
 fraction above replaced by its inverse.

$$
B_{\e}=t_0 -\frac{t_0 \Bigl(1-\e/\chi_\mu(\lambda')  -
\frac{\chi_\mu(\eta') + \e}{-\chi_\mu(\lambda')-\e}\Bigr)
\Bigl(-\chi_\mu(\la')\Bigr)  }
{\Bigl(1+\frac{\chi_\mu(\eta') + \e}{-\chi_\mu(\lambda')-\e}\Bigr)
\Bigl(-\chi_\mu(\la')+\e\Bigr)}.
$$

\end{proposition}

Here $A_\e$ bounds Hausdorff dimension of the irregular part and
$B_\e$ bounds Hausdorff dimension of the regular non-Lipschitz part.

For $\e\approx 0$ the number $A_\e$ is bigger. On the other end, for
$\frac{\chi_\mu(\eta') + \e}{-\chi_\mu(\lambda')-\e}$ almost 1
the number $B_\e$ dominates. Optimum is in between.

\medskip

\begin{proof}
First we prove the estimate \eqref{eq:HDNL} for $B_\e$. We rely on Section 5: Proof of Theorem \ref{packing}, Step 1.
The coefficient $\alpha$ is not needed,
since $NL^w$ is a local property and  overlappings of remote cylinders do
not count (see Proof of Theorem \ref{packing}, Step 3).


We obtain the uniform estimate for every $(\la,\e,m)$-backward \underline{regular} $H_m$, with $n$ satisfying \eqref{nm}
Its "\underline{contaminated part}"
can be estimated as follows, see \eqref{contaminated_bound} and \eqref{eq:h*h},
$$
\mu (\cV_n(H_m)\cap H_m)/\mu(H_n) \le C(n)
 \exp (n h^* \frac{\chi_\mu(\eta) + \e }{ -\chi_\mu(\la) - \e}) \cdot \exp ( -n (h^*-\e))   \le
$$
\begin{equation}\label{meas}
C(n) \exp  n h^* \bigl(   \frac{\chi_\mu(\eta) + \e }{ -\chi_\mu(\la) - \e} - 1 + \frac{\e}{h^*}      \bigr) ,
\end{equation}
where $C(n)$ grows sub-exponentially. We used here, as already e.g. in \eqref{contaminated_bound}, the fact
that $\Const^{-1}\mu(H_m \cap V_n)/ \mu(H_m) \mu(V_n) <\Const$ following from Gibbs property of $\mu$.

Now by summing over regular $H_m$ with weights $\mu (H_m)$ we get the same estimate for $\mu \bigl( \bigcup_{H_m}(\cV_n(H_m)\cap H_m)\bigr)$
and by the $f$-invariance of $\mu$ the same estimate for $\bH'(n):= f^n \bigl(\bigcup_{H_m}(\cV_n(H_m)\cap H_m)\bigr)$
for $n=n(\e,m)$,  built of cylinders $H_{m+n}$ .

\smallskip

Now we shall translate the measure estimate above for all $m, n(\e,m)$, to an estimate of Hausdorff dimension.

Denote $\mu(\bH'(n))$ by $\ov\mu_n$.
By Gibbs property of $\mu=\mu_{t_0}$ and using normalized restrictions $\mu_n:=\mu|_{\bH'(n)}/\ov\mu_n$,
considering conditional measures 
on $W^s$ (not changing notation), we get for each $H_{n+m}\subset \bH'(n)$ and $p$ in it

\begin{equation}\label{deviation-HD}
\mu_n (H_{n+m}) \ge \Const(\la_{n+m}^- (p))^{t_0}  / \ov\mu_n 
\ge \exp ((n+m)(\chi_\mu(\la')-\e) (t_0 - \vartheta_n),
\end{equation}
where $\vartheta_n:=\frac{\log\ov\mu_n}{(n+m)(\chi_\mu(\la')-\e)}$.
with $m$ expressed by $n$ maximal possible to satisfy \eqref{nm}.
So, for any $x\in S^1$,
\begin{equation}
\sum_{H_{n+m}\subset\bH'(n)} \diam (H_{n+m} \cap W^s_x)^{(t_0-\vartheta_n)\kappa} \le
\sum_{H_{n+m}\subset\bH'(n)} \mu_n (H_{n+m})=1,
\end{equation}
where $\kappa=\frac{\chi_\mu(\la') -e}{\chi_\mu(\la')}$.
By an arbitrarily small change of $\kappa$ we can assure the bound by 1 replaced
by numbers tending exponentially  to 0 as $m\to\infty$, allowing summing over $m$.

So, taking $\e\to 0$,
for $\vartheta:=\liminf_{n\to\infty}\vartheta_n$,
$$
\HD(NL^w\cap W^s_x) \le \HD(\limsup_{n\to\infty} \bH'(n)) \le  \vartheta.
$$
This, after simple calculations, yields the estimate by $B_\e$ in our Proposition.
 For an explanation of the structure of $L^s$ complementary to the set $NL^w \subset \limsup_{n\to\infty} \bH'(n)$,(with $2n$ there in place of $m+n$ here) see Remark {rem:5.3}.


\medskip

Now we estimate the irregular (backward) part,
related to $A_\e$.
For this we define similarly to $\bH_i$ in Section 5 but with $k=0$ and $\a=1$
We add $\la,\eta$ and $\e$ in the notation of $\bH_i$. We consider

$\bH_1(\la,\e,m, \irreg)\subset \bigcup H_{m+n}$ and

$\bH_2(\eta,\e,n, \irreg)\subset \bigcup H_{m+n}$.

\noindent We define also as in Proof of Theorem \ref{packing}, Step 3, the set

$\bH_3(\xi,\e, n, \irreg)\subset \bigcup H_{m+n}$ for $\xi=\la,\nu$.

Applying Lemma \ref{large_deviations} for $\psi=\log\la'$, replacing $m$ by $n+m$, we get for $\bH_1$
$$
\HD (\limsup \bH_1(\la,\e,m, \irreg)) \le t_0 - (I(\log\la',\e)/\chi_\mu(\la')) /
(1+\frac{-\chi_\mu(\lambda')-\e}{\chi_\mu(\eta') + \e}),
$$
for $\bH_2$, replacing $n$ by $n+m$
$$
\HD (\limsup \bH_2(\eta,\e,n, \irreg)) \le t_0 - (I(-\log\eta',\e)/\chi_\mu(\la'))/
(1+\frac{\chi_\mu(\eta') + \e}{-\chi_\mu(\lambda')-\e})
$$
and for $\bH_3$
$$
\HD (\limsup \bH_3(\la,\e, n, \irreg)) \le t_0 - (I(\log\la',\e)/\chi_\mu(\la')) /
(1+\frac{\chi_\mu(\eta') + \e}{-\chi_\mu(\lambda')-\e}).
$$

\end{proof}

\subsection{Generalization to 1-dimensional expanding attractors}

\

\bigskip

All the theorems in this paper hold also for hyperbolic expanding attractors in dimension 3 with 1-dimensional unstable manifolds, non-uniformly thin (see definition in Section \ref{Introduction}) and satisfying the transversality assumption,  our solenoids are example of. The only exception is thr theorem on singularity of Hausdorff measures Theorem \ref{Hmeas}, where the assumption that for some $p,q\in \La$ there is a non-empty intersection of projections
$\hat W^u(p)$ and $\hat W^u(q)$ is needed. For our solenoids it holds automatically but for extensions to $\R^3$ of
say Plykin or DA attractor it is not so. See \cite{Robinson}.

Proofs are the same since these attractors are extensions of expanding maps on branched 1-manifolds and Markov coding can be used.

\subsection{More on solenoids -- coordinates}\label{solenoids}

In fact Theorems A-D hold for
\begin{equation}\label{par}
f(x,y,z):= (\eta(x,y,z), \lambda(x,y,z) + u(x), \nu(x,y,z ) + v(x)).
\end{equation}
of class $C^{1+\e}$, injective, such that $f(\cl M)\subset M$,
satisfying
$\lambda(x,0,0)=\nu(x,0,0)=0$, with hyperbolic attractor $\Lambda$,
and satisfying transversality,  the non-conformal form
more general than $f$ in the triangular  in \eqref{triangular}.

Indeed, we are interested in non-conformal solenoid, so we assume that
 the tangent bundle on $M$, or at least on $\Lambda$, splits into $T_\Lambda M=E^u \oplus E^{s}$, $Df$ invariant, where $E^s$,
 the stable one, splits further into weak stable and strong stable
 $T_\La M= E^u \oplus E^{ws} \oplus E^{ss}$, or at least $E^s$ contains a strong stable $E^{ss}$.  Note that $E^s$  is dynamically defined on the whole $M$, not only on $\Lambda$, by $E^s(p):=\lim Df^{-n} (C^s(f^n(p)))$ where $C^s$ denote a stable cone taken equal to a cone at a point in $\Lambda$ near $f^n(p)$. Similarly one proves that the bundle $E^s$ on $M$ is integrable to a
stable foliation $\cW^s$ of $M$. As having codimension 1 it is  $C^{1+\e}$,
  see \cite{PRF}.  Therefore under an appropriate $C^{1+\e}$ change of coordinates  it becomes the foliation of $M$ by vertical discs $W^s_x=\{x\}\times \D$.

  Also strong stable foliation $\cW^{ss}$ (of the whole $M$ as obtained as a limit from the future) can be made consisting of vertical intervals, that is  with $x,y$ constant.
  This foliation is known to be $C^{1+\e}$ in $\cW^s$, see \cite{Brown}, hence after a change of coordinates so that it becomes vertical, our diffeomorphism is $C^{1+\e}$ in each $W^s$. We do not know however what is the smoothness of $f$ in the new coordinates in the whole $M$.

  Therefore to deduce this general case from our triangular case by change of coordinates we just assume
  $\cW^{ss}$  is  $C^{1+\e}$  in $M$. A question stays open whether this assumption is needed, i.e. whether we really use $f$ being $C^{1+\e}$ in the triangular coordinates.

\medskip

The following completes the topological picture. Suppose $f$ is already in the triangular form.

\begin{lemma}\label{weak-stable}
There exists on $M$ a change of coordinates

\noindent $\Psi(x,y,z)=(x,y,\psi(y,z)$, bi-Lipschitz continuous, such that the foliation into the sets $x,z$ constant is invariant and its $\Psi^{-1}$-image is a central stable foliation
$\cW^{sc}$ with leaves $C^1$ smooth.
\end{lemma}
\begin{proof}
Extend $f$ to $\tilde f :S^1 \times \R^2\to S^1 \times \R^2$ so that $\lambda$ and $\nu$ are linear with respect to $y$ and $z$ respectively, far from $M$.

Next find $\cW^{sc}$ as a limit of $\tilde f^n (\cW_y)$, where $\cW_y$ is the foliation
of $M$ into the intervals $x,z$ constant. By bounded distortion one gets Lipschitz property of the limit and
in particular a true foliation (leaves do no glue partially to each other in the limit).
\end{proof}

Our $Df$ in these coordinates would be diagonal which would ease estimates.  Unfortunately this central stable foliation and therefore $f$ in the new coordinates seems usually not $C^{1+\e}$.

\subsection{Summary of our strategies}\label{strategies}

The key objects in the paper are "rectangles" being intersection of horizontal and vertical strips $\hat H_m$ and $\hat V_n$, "cylinders" of level $m$ and $n$ respectively, projections to the plane $(x,y)$ of tubes and thickened discs. Such Markov rectangles are basic objects in hyperbolic dynamics.

Horizontal strips can intersect transversally other horizontal strips. An issue is to estimate
how large part of any horizontal strip is intersected, "contaminated", by other horizontal strips, measured in a number of
contaminated (with margins) rectangles.  The tool is going backward by $f^{-m}$ or forward by $f^n$ to large scale, so that the rectangles
become full (that is over $[0,2\pi]$) horizontal strips and results do not depend on sections by stable discs $W^s$.
We distinguish Birkhoff irregular sets among  full unstable manifolds (over $[0,2\pi]$)
and prove they have stable SRB-measure 0 and even Hausdorff dimension in each $W^s$ less than the dimension of $\La\cap W^s$. We estimate also the size of the contaminated set of Birkhoff regular unstable manifolds. In each section the choise
of $m$ to $n$ (or vice versa) and auxiliary $k$ is different, depending on our needs.

\bibliographystyle{amsplain}

\end{document}